\documentclass[10pt, leqno]{amsart}
\usepackage{amsfonts,amssymb}

\usepackage{amsmath}
\usepackage{amsthm}
\usepackage{amsrefs}
\usepackage{qsymbols}
\usepackage{latexsym}
\usepackage{chngcntr}
\usepackage{paralist}
\usepackage{mathtools}
\usepackage{esint}
\usepackage[hidelinks]{hyperref}

\newtheorem{theorem}{Theorem}[section]

\newtheorem{lemma}[theorem]{Lemma}
\newtheorem{proposition}[theorem]{Proposition}
\newtheorem{corollary}[theorem]{Corollary}
\theoremstyle{definition}
\newtheorem{definition}[theorem]{Definition}
\newtheorem{remark}[theorem]{Remark}

\newtheorem{assumption}[theorem]{Assumption}

\newcommand{\IR}{\mathbb{R}}
\newcommand{\IC}{\mathbb{C}}
\newcommand{\IN}{\mathbb{N}}
\newcommand{\IZ}{\mathbb{Z}}

\newcommand{\cO}{\mathcal{O}}

\newcommand{\cP}{\mathcal{P}}

\newcommand{\cA}{\mathcal{A}}
\newcommand{\cB}{\mathcal{B}}
\newcommand{\cC}{\mathcal{C}}
\newcommand{\cL}{\mathcal{L}}

\newcommand{\cV}{\mathcal{V}}

\renewcommand{\L}{\mathrm{L}}
\newcommand{\C}{\mathrm{C}}

\renewcommand{\H}{\mathrm{H}}

\renewcommand{\S}{\mathrm{S}}

\renewcommand{\P}{\mathrm{T}_0}

\newcommand{\Lloc}{\L_{\mathrm{loc}}}

\newcommand{\fa}{\mathfrak{a}}

\newcommand{\ind}{{\mathbf{1}}}

\newcommand{\bigdot}{\boldsymbol{\cdot}}

\newcommand{\ii}{\mathrm{i}}
\renewcommand{\d}{\mathrm{d}}
\newcommand{\eps}{\varepsilon}
\newcommand{\loc}{\mathrm{loc}}
\renewcommand\Re{\operatorname{Re}}

\newcommand{\cube}{\scalebox{1.1}{$\square$}}
\newcommand{\smallcube}{\scalebox{0.8}{$\square$}}
\newcommand{\Lop}{\mathcal{L}}

\newcommand{\divergence}{\operatorname{div}}
\newcommand{\esssup}{\mathrm{ess sup}}

\DeclareMathOperator{\supp}{supp}

\DeclareMathOperator{\tr}{tr}

\DeclareMathOperator{\Id}{Id}

\DeclareMathOperator{\dom}{\mathcal{D}}

\newcommand{\DyadicMax}{\cA}
\numberwithin{equation}{section}

\title[Kato's Property for Generalized Stokes Operators]{On Kato's Square Root Property for the Generalized Stokes Operator}

\author{Luca Haardt}
\author{Patrick Tolksdorf}

\address{Karlsruhe Institute of Technology, Department of Mathematics, 76131 Karlsruhe, Germany}

\email{luca.haardt@kit.edu}
\email{patrick.tolksdorf@kit.edu}

\subjclass[2010]{}
\date{\today}
\thanks{}

\begin{document}
\begin{abstract}
We establish the Kato square root property for the generalized Stokes operator on $\IR^d$ with bounded measurable coefficients. More precisely, we identify the domain of the square root of $Au \coloneqq - \divergence(\mu \nabla u) + \nabla \phi$, $\divergence(u) = 0$, with the space of divergence-free $\H^1$-vector fields and further prove the estimate $\| 
A^{1/2} u \|_{\L^2} \simeq \| \nabla u \|_{\L^2}$. As an application we show that $A^{1/2}$ depends holomorphically on the coefficients $\mu$. Besides the boundedness and measurablility as well as an ellipticity condition on $\mu$, there are no requirements on the coefficients.
\end{abstract}
\maketitle

\section{Introduction}

\noindent At the beginning of the 1960s, Tosio Kato asked, whether, given two Hilbert spaces $V \subseteq H$ with $V$ being dense in $H$, the domain of the square root of the maximal accretive operator $L$ in $H$ associated to a closed and sectorial sesquilinear form $\fa \colon V \times V \to \IC$ always coincides with the form domain, i.e., whether $\dom(L^{1/2}) = V$. In the subsequent years it became clear by counterexamples of Lions~\cite{Lions} and McIntosh~\cite{McIntosh} that such a result does not hold in this generality. Lions already specified the question to elliptic operators $L = - \divergence (\mu \nabla \bigdot)$ with bounded measurable coefficients which became known as the Kato square root problem. It was eventually resolved in the whole space to the affirmative by Auscher, Hofmann, Lacey, McIntosh and Tchamitchian~\cite{AHLMcT} in 2002. \par 
In the last two decades there has been a surge of interesting results that revealed, for instance, a deep connection of the square root property to boundary value problems~\cites{Auscher_Egert, AAAHK, Auscher_Axelsson_McIntosh, AA, AM, AusSta, HMiMo}, operator-adapted function spaces~\cites{Auscher_Egert, Hofmann_Mayboroda_McIntosh, AusSta}, maximal regularity and quasilinear PDE~\cites{AA1, ABHR} and many more~\cites{Auscher, Lashi, HKMP, HPR, Bechtel, Egert}. Since the restriction from densely defined, closed and sectorial sesquilinear forms to elliptic operators in divergence form is quite drastic, there has always been the question for which other operators this square root property holds. \par
For elliptic operators the validity of the square root property was extended to situations including rough boundary geometries and mixed boundary conditions~\cites{AKM_mixed, Egert_Haller_Tolksdorf, Bechtel_Egert_Haller} as well as to operators on submanifolds~\cite{Morris}. Using the theory of Muckenhoupt weights, it was extended to elliptic operators with degenerate coefficients~\cites{C-UR1, C-UR2, C-UMR}. It was further established for Schr\"odinger operators~\cite{Bailey} and parabolic operators~\cite{Auscher_Egert_Nystroem, Nystroem}. \par
The purpose of this paper is to enrich the class of operators sharing the square root property by a nonlocal operator that naturally arises in the theory of fluid mechanics. More precisely, we consider the generalized Stokes operator in the whole space, which is formally given by
\begin{align}
\label{Eq: Formal Stokes operator}
Au \coloneqq - \divergence(\mu \nabla u) + \nabla \phi, \quad \divergence(u) = 0
\end{align}
and assume that the coefficients $\mu$ are bounded and measurable, and satisfy a G\r{a}rding inequality. The sesquilinear form to which $A$ is associated is given by
\begin{align}
\label{Eq: Sesquilinear form}
 \fa : \H^1_{\sigma} (\IR^d) \times \H^1_{\sigma} (\IR^d) \to \IC, \quad (u , v) \mapsto \sum_{\alpha , \beta , i , j = 1}^d \int_{\IR^d} \mu_{\alpha \beta}^{i j} \partial_{\beta} u_j \overline{\partial_{\alpha} v_i} \, \d x,
\end{align}
where $\H^1_{\sigma} (\IR^d) \coloneqq \{ u \in \H^1 (\IR^d ; \IC^d) \colon \divergence(u) = 0 \}$. Thus, the square root property for the generalized Stokes operator asks for the identification of $\dom(A^{1/2})$ with $\H^1_{\sigma} (\IR^d)$. \par
Even though the sesquilinear form in~\eqref{Eq: Sesquilinear form} has the same form as for elliptic operators in divergence form, it has the crucial difference expressed in the fact that the form domain is only given by $\H^1_{\sigma} (\IR^d)$. This results in the appearance of the pressure gradient in~\eqref{Eq: Formal Stokes operator} and, in stark contrast to elliptic operators, has the effect that $A$ is \textit{nonlocal}. The locality of elliptic operators is a key-property for the derivation of one of the most important tools used in the proof of the square root property commonly known as \textit{off-diagonal estimates}. For the resolvent of the generalized Stokes operator such off-diagonal estimates are not known. Recently, the second author established a nonlocal version of such estimates with polynomial decay~\cite{Tolksdorf_off-diagonal} but, so far, the order of decay is too small and thus insufficient for our purposes. We will improve the order of decay of these estimates by an iterative procedure and thereby attain a sufficient order of decay. Afterwards we will show how to adapt the proof of the square root property for elliptic operators to these new nonlocal estimates.

Let us state the main result of this article in detail. We start with our assumptions on the coefficients.

\begin{assumption}
\label{Ass: Coefficients}
The coefficients $\mu = (\mu_{\alpha \beta}^{i j})_{\alpha , \beta , i , j = 1}^d$ with $\mu_{\alpha \beta}^{i j} \in \L^{\infty} (\IR^d ; \IC)$ for all $1 \leq \alpha , \beta , i , j \leq d$ satisfy for some $\mu_{\bullet} , \mu^{\bullet} > 0$ the inequalities
\begin{align*}
 \Re \sum_{\alpha , \beta , i , j = 1}^d \int_{\IR^d} \mu^{i j}_{\alpha \beta} \partial_{\beta} u_j \overline{\partial_{\alpha} u_i} \, \d x \geq \mu_{\bullet} \| \nabla u \|_{\L^2}^2 \qquad (u \in \H^1 (\IR^d ; \IC^d))
\end{align*}
and
\begin{align*}
 \esssup_{x \in \Omega} \|\mu(x)\|_{\Lop(\IC^{d \times d})} \leq \mu^{\bullet}.
\end{align*}
\end{assumption} 

The generalized Stokes operator $A$ is realized on $\L^2_{\sigma} (\IR^d)$ via the sesquilinear form in~\eqref{Eq: Sesquilinear form}. Thus, $A$ is given by $A u \coloneqq f$, where $u \in \dom(A)$ and $f$ are associated via
\begin{align*}
 \dom(A) := \bigg\{ u \in \H^1_{\sigma} (\IR^d) : \, \exists f \in \L^2_{\sigma} (\IR^d) \text{ such that } \forall v \in \H^1_{\sigma} (\IR^d) \text{ it holds } \fa (u , v) = \int_{\IR^d} f \cdot \overline{v} \, \d x  \bigg\}.
\end{align*}

The main result of this article concerns a characterization of the domain of the square root of $A$ as the space of divergence-free $\H^1$-vector fields. We stress, that neither symmetry nor regularity of the coefficients is assumed.

\begin{theorem}
\label{Thm: Kato}
Let $\mu$ satisfy Assumption~\ref{Ass: Coefficients}. Then $A$ has the square root property, i.e., we have 
 that $\dom(A^{1 / 2}) = \H^1_{\sigma} (\IR^d)$ and
\begin{align*}
 C^{-1} \| \nabla u \|_{\L^2} \leq \| A^{1/2} u \|_{\L^2} \leq C \| \nabla u \|_{\L^2} \qquad (u \in \H^1_{\sigma} (\IR^d))
\end{align*}
where $C > 0$ only depends on $d$, $\mu_{\bullet}$ and $\mu^{\bullet}$.
\end{theorem}

A particular application that Kato had in mind was the Lipschitz dependency of the square roots $A^{1/2}$ with respect to the coefficients $\mu$ measured in the $\L^{\infty}$-topology, which is proven via holomorphy. To formulate such properties and statements let us emphasize the coefficients in the notation of $A$ by writing $A_{\mu}$ for the generalized Stokes operator with coefficients $\mu$. 

\begin{theorem}
\label{Thm: Holomorphic dependence}
The set $\cO \coloneqq \{ 
\mu \colon \mu \text{ satisfies Assumption~\ref{Ass: Coefficients} for some $\mu_{\bullet} , \mu^{\bullet} > 0$} \}$ is open in the $\L^{\infty}$-topology and the map
    \begin{align*}
        \cO \to \mathcal{L}(\H^1_\sigma(\IR^d),\L^2_\sigma(\IR^d)), \quad \mu \mapsto A_{\mu}^{1/2}
    \end{align*}
    is holomorphic, i.\@e., for every $\mu\in\cO$ and $M\in \L^\infty(\IR^d;\cL(\IC^{d \times d}))$ there exists $r>0$ such that the map
    \begin{align*}
        \{z\in \IC : |z|<r\} \to \mathcal{L}(\H^1_\sigma(\IR^d),\L^2_\sigma(\IR^d)), \quad z \mapsto A^{1/2}_{\mu+zM}
    \end{align*}
    is holomorphic.
\end{theorem}

This kind of holomorphy of $\mu \mapsto A_{\mu}^{1/2}$ implies the above-mentioned local Lipschitz property for small perturbations which reads as follows.

\begin{corollary}
\label{cor: Lipschitz estimate}
Let $\mu_1 \in \cO$. Then there exists $\delta>0$ and a constant $C>0$ depending only on $(\mu_1)_\bullet, (\mu_1)^\bullet, d$ and $\delta$ such that for all $\mu_2\in \cO$ with $\| \mu_2 - \mu_1 \|_{\L^{\infty} (\IR^d ; \cL(\IC^{d \times d}))}<\delta$ we have
\begin{align*}
 \| A_{\mu_2}^{1/2} - A_{\mu_{1}}^{1/2} \|_{\cL(\H^1_{\sigma} , \L^2_{\sigma})} \leq C \| \mu_2 - \mu_1 \|_{\L^{\infty} (\IR^d ; \cL(\IC^{d \times d}))}.
\end{align*}
\end{corollary}

Let us mention that the generalized Stokes operator arises in the theory of fluid mechanics as the linearization of non-Newtonian fluids, see, e.g.,~\cite[Sec.~12.1]{Pruss_Simonett}. In this context, symmetric coefficients are of particular interest and the square root property can be established by \textit{Kato's second representation theorem}~\cite{Kato_perturbation} using pure abstract reasoning. However, already for elliptic operators in divergence form, there is no known proof that establishes the Lipschitz estimate in Corollary~\ref{cor: Lipschitz estimate} for symmetric coefficients by remaining in the class of operators with symmetric coefficients. All existing proofs, such as the one presented here, use in a crucial way the square root property for operators with nonsymmetric coefficients. Thus, this Lipschitz estimate can be regarded as nontrivial even if one is only interested in symmetric coefficients.  

The outline of this paper is as follows. In Section~\ref{Sec: L2-theory of the generalized Stokes resolvent} we collect elementary $\L^2$-estimates on the resolvent of the generalized Stokes operator and then establish the above-mentioned nonlocal off-diagonal estimates of polynomial order. Afterwards we begin to adapt the proof of the square root property for elliptic operators as it was presented in~\cite{ISEM}. In Section~\ref{Sec: Reduction to a square function estimate} we shortly recapitulate the reduction to a square function estimate which is then established in Sections~\ref{Sec: Principle part approximation} and~\ref{Sec: The Carleson measure argument}. Section~\ref{Sec: Principle part approximation} presents how the classical principle part approximation can be adapted to these nonlocal off-diagonal estimates of polynomial order, leading --- as in the elliptic situation --- to a reduction to a Carleson measure estimate. This is established in Section~\ref{Sec: The Carleson measure argument} where a $T(b)$-type test function has to be constructed. Compared to elliptic operators there will be an additional constraint in this construction as it has to cope with two ingredients at the same time --- divergence-freeness and locality. Here, a suitable Bogovski\u{\i} correction will be crucial. In the final Section~\ref{Sec: Holomorphic dependence} we establish Theorem~\ref{Thm: Holomorphic dependence} and Corollary~\ref{cor: Lipschitz estimate}.

\subsection*{Acknowledgments}

We would like to thank Moritz Egert for suggesting to us to look into the holomorphic dependencies as well as the Lipschitz estimates as a possible application of the square root property. The first author was supported by \textit{Studienstiftung des deutschen Volkes}.

\subsection*{Notation}

Throughout this article, the dimension is denoted by $d$ and satisfies $d \geq 2$. We will write $\IN$ for the set of all positive integers and $\IN_0$ for the set of all non-negative integers. The norms of $\IC$ and $\IC^d$ will be denoted by $\lvert \, \bigdot \, \rvert$, all other norms will be labeled accordingly. We will use the euclidean norm on $\IC^d$ and $\IC^{d \times d}$. The open ball centered in $x \in \IR^d$ and with radius $r > 0$ is denoted by $B(x , r)$. Given an open ball $B$ and $\ell \in \IN_0$, the $\ell$-th dyadic annulus around $B$ is defined by
\begin{align*}
 C_{\ell} (B) \coloneqq B(x , 2^{\ell + 1} r) \setminus \overline{B(x , 2^{\ell} r)} \quad \text{if} \quad \ell \geq 1 \quad \text{and} \quad C_0 (B) \coloneqq 2B.
\end{align*}
Given a cube $Q \subseteq \IR^d$ we denote its sidelength by $\ell (Q)$. As for balls, we define 
\begin{align*}
 C_{\ell} (Q) \coloneqq 2^{\ell + 1} Q \setminus \overline{2^{\ell} Q} \quad \text{if} \quad \ell \geq 1 \quad \text{and} \quad C_0 (Q) \coloneqq 2Q.
\end{align*}

In estimates it will be convenient to write $\alpha \lesssim \beta$ or $\alpha\gtrsim \beta$ if there exists $C > 0$, depending only on parameters not at stake, such that $\alpha \leq C \beta$ or $C\alpha \geq \beta$. We will write $\alpha \simeq \beta$ if $\alpha \lesssim \beta$ and $\alpha \gtrsim \beta$. In some situations it will be more handy to keep the notation $\alpha \leq C \beta$ or $C\alpha \geq \beta$. In this case $C$ is generic. \par
The characteristic function of a set $A \subseteq \IR^d$ is denoted by $\ind_A$. If $A$ is measurable we denote its Lebesgue measure by $\lvert A \rvert$. In the case $0 < \lvert A \rvert < \infty$, we define the mean value of $f \in \L^1_{\loc} (\IR^d)$ as
\begin{align*}
 (f)_A \coloneqq \fint_A f \, \d x \coloneqq \frac{1}{\lvert A \rvert} \int_A f \, \d x.
\end{align*}

Important function space of divergence-free vector fields are given by
\begin{align*}
 \L^2_{\sigma} (\IR^d) &\coloneqq \{ f \in \L^2 (\IR^d ; \IC^d) \colon \divergence(f) = 0 \text{ in sense of distributions} \} \\
 \H^1_{\sigma} (\IR^d) &\coloneqq \{ f \in \H^1 (\IR^d ; \IC^d) \colon \divergence(f) = 0 \}.
\end{align*}
Recall that 
\begin{align*}
 \C_{c , \sigma}^{\infty} (\IR^d) \coloneqq \{ 
\varphi \in \C_c^{\infty} (\IR^d ; \IC^d) \colon \divergence(\varphi) = 0 \}
\end{align*}
is dense in $\L^2_{\sigma} (\IR^d)$ as well as in $\H^1_{\sigma} (\IR^d)$, see Lemma~II.2.5.4 and Lemma~II.2.5.5 in~\cite{Sohr}. The space of all bounded and antilinear functionals from $\H^1_{\sigma} (\IR^d) \to \IC$ is denoted by $\H^{-1}_{\sigma} (\IR^d)$. Given $u \in \H^1 (\IR^d ; \IC^d)$, we regard its gradient as the matrix given as the transpose of the Jacobian of $u$, i.e., $\nabla u = (\partial_{\alpha} u_i)_{\alpha , i = 1}^d.$ \par
For $\alpha \in (0 , 1)$ the fractional Sobolev space with differentiability $\alpha$ will be denoted by $\H^{\alpha}(\IR^d)$. We write its norm in integral form 
\begin{align*}
 \| f \|_{\H^{\alpha}} \coloneqq \| f \|_{\L^2} + \| f \|_{\dot \H^{\alpha}} \quad \text{where} \quad \| f \|_{\dot \H^{\alpha}} \coloneqq \bigg(\iint_{\IR^d \times \IR^d} \frac{\lvert f(x) - f(y) \rvert^2}{\lvert x - y \rvert^{d + 2 \alpha}} \, \d x \, \d y \bigg)^{\frac{1}{2}}.
\end{align*}
For $f \in \H^{\alpha} (\IR^d)$ we will need to calculate a counterpart of the homogeneous $\dot \H^{\alpha}$-seminorm on a measurable set $A \subseteq \IR^d$ which we write as
\begin{align*}
 \| f \|_{\dot \H^{\alpha} (A)} \coloneqq \bigg(\iint_{A \times A} \frac{\lvert f(x) - f(y) \rvert^2}{\lvert x - y \rvert^{d + 2 \alpha}} \, \d x \, \d y \bigg)^{\frac{1}{2}}.
\end{align*}

\section{$\L^2$-theory of the generalized Stokes resolvent}
\label{Sec: L2-theory of the generalized Stokes resolvent}


\noindent This section is devoted to properties of the generalized Stokes operator. It contains elementary properties that directly follow from classical form theory. The main result of this section is Proposition~\ref{Prop: off-diagonal type estimates for div} which can be seen as nonlocal off-diagonal estimates for the generalized Stokes resolvent. We start by introducing an extension of the generalized Stokes operator in $\H^{-1}_{\sigma} (\IR^d)$.

\begin{definition}
    The \textit{weak (generalized) Stokes operator} $\cA : \H^1_{\sigma} (\IR^d) \subseteq \H^{-1}_{\sigma} (\IR^d) \to \H^{-1}_{\sigma} (\IR^d)$ associated to the sesquilinear form $\fa$ is defined by
    \begin{align*}
        \langle \cA u , v \rangle_{ \H^{-1}_{\sigma},\H^1_{\sigma}} \coloneqq \fa (u , v) \qquad (u,v\in \H^1_{\sigma} (\IR^d)).
    \end{align*}
\end{definition}

Observe that $u\in \dom(A)$ if and only if $\cA u \in \L^2_{\sigma}(\IR^d)$. Hence, $\cA$ is an extension of $A$. For later reasoning it will be important to deduce a representation of $\cA$ as a product of the operators $\divergence$, $\mu$ and $\nabla$. For fixed $x$, the coefficients may be  interpreted as a linear operator from $\IC^{d \times d}$ to $\IC^{d \times d}$ via
\begin{align*}
 \mu(x) G \coloneqq \Big( \sum_{j , \beta = 1}^d \mu_{\alpha \beta}^{i j} (x) G_{\beta j} \Big)_{\alpha, i = 1}^d \qquad (G \in \IC^{d \times d}).
\end{align*}
If one defines the dot-product of two matrices $G , H \in \IC^{d \times d}$ as $G \cdot H \coloneqq \sum_{\alpha , i = 1}^d G_{\alpha i} H_{\alpha i}$, then the sesquilinear form defined in~\eqref{Eq: Formal Stokes operator} can be written as
\begin{align}
\label{Eq: Sesquilinear form without sums}
 \fa(u,v) = \int_{\IR^d} \mu \nabla u \cdot \overline{\nabla v} \, \d x.
\end{align}
Given $F \in \L^2 (\IR^d ; \IC^{d \times d})$, its weak divergence is defined, as usual, as an element in $\H^{-1} (\IR^d ; \IC^d) \coloneqq (\H^1 (\IR^d ; \IC^d))^{\prime}$ via
\begin{align*}
 \langle \divergence (F) , w \rangle_{\H^{-1} , \H^1} \coloneqq - \sum_{\alpha , i = 1}^d \int_{\IR^d} F_{\alpha i} \overline{\partial_{\alpha} w_i} \, \d x \qquad (w \in \H^1 (\IR^d ; \IC^d)).
\end{align*}
To interprete the function $v \in \H^1_{\sigma} (\IR^d)$ in~\eqref{Eq: Sesquilinear form without sums} as an element in $\H^1 (\IR^d ; \IC^d)$ we introduce the canonical inclusion $\iota : \H^1_{\sigma} (\IR^d) \to \H^1 (\IR^d ; \IC^d)$. If $\cP \coloneqq \iota' : \H^{-1} (\IR^d ; \IC^d) \to \H^{-1}_{\sigma} (\IR^d)$ denotes its adjoint, then for $u , v \in \H^1_{\sigma} (\IR^d)$ we have
\begin{align*}
            \langle \cA u , v \rangle_{\H^{-1}_{\sigma},\H^1_{\sigma}} = \int_{\IR^d} \mu \nabla u \cdot \overline{\nabla \iota(v)} \, \d x = \langle -\divergence(\mu \nabla u) , \iota (v) \rangle_{\H^{-1} , \H^1} = \langle -\cP \divergence(\mu \nabla u) , v \rangle_{\H^{-1}_{\sigma} , \H^1_{\sigma}}.
\end{align*}
We record this representation in the following lemma.

\begin{lemma}
\label{Lem: Product structure}
We have
\begin{align*}
 \cA u = - \cP \divergence(\mu \nabla u) \qquad (u \in \H^1_{\sigma} (\IR^d)).
\end{align*}
\end{lemma}

\begin{remark}
The representation in Lemma~\ref{Lem: Product structure} was already proven by Mitrea and Monniaux for the classical Stokes operator (with $\divergence (\mu \nabla u)$ replaced by $\Delta u$) in bounded Lipschitz domains, see~\cite[Prop.~4.5]{Mitrea_Monniaux}.
\end{remark}

The control on the essential supremum of the $\cL(\IC^{d \times d})$-norm of $\mu$ in Assumption~\ref{Ass: Coefficients} can precisely be understood as $\lvert \mu (x) G \rvert_{\IC^{d \times d}} \leq \mu^{\bullet} \lvert G \rvert_{\IC^{d \times d}}$ for $G \in \IC^{d \times d}$ and $x\in\IR^d$. Thus
\begin{align*}
 \lvert \fa (u , v) \rvert \leq \mu^{\bullet} \| \nabla u \|_{\L^2} \| \nabla v \|_{\L^2} \qquad (u , v \in \H^1 (\IR^d ; \IC^{d})),
\end{align*}
so that $\fa$ is, in particular, a bounded sesquilinear form on $\H^1_{\sigma} (\IR^d)$. The G\r{a}rding-type estimate in Assumption~\ref{Ass: Coefficients} exactly means that $\fa$ is coercive. As a consequence, the uniform resolvent bounds in the next proposition, in particular the sectoriality of the generalized Stokes operator in $\L^2_{\sigma} (\IR^d)$, follow from classical form theory involving the Lax--Milgram lemma. The proof will be omitted.

For the rest of this section, we will denote a sector in the complex plane with opening angle $2 \theta$, $\theta \in (0 , \pi)$, around the positive real line as
\begin{align*}
 \S_{\theta} \coloneqq \{ z \in \IC \setminus \{ 0 \} : \lvert \arg (z) \rvert < \theta \}.
\end{align*}

\begin{proposition}
\label{Prop: L2 resolvent bounds}
    Let $\mu$ satisfy Assumption~\ref{Ass: Coefficients}. Then there exists $\omega\in (\pi/2,\pi)$ depending only on $d$, $\mu_{\bullet}$ and $\mu^{\bullet}$ such that $\S_\omega\subseteq \rho(-A) \cap \rho(- \cA)$ and for all $\theta\in (0,\omega)$ there exists $C>0$ such that for all $\lambda \in \S_\theta$ and all $f\in \L^2_\sigma(\IR^d)$ we have
    \begin{align*}
        |\lambda|\|(\lambda + A)^{-1}f\|_{\L^2} + |\lambda|^\frac{1}{2}\|\nabla(\lambda + A)^{-1}f\|_{\L^2} \leq C\|f\|_{\L^2}.
    \end{align*}
    Moreover, there exists $C>0$ such that for all $\lambda \in \S_\theta$ and all $F\in \L^2(\IR^d;\IC^{d\times d})$ we have
    \begin{align*}
        |\lambda|^\frac{1}{2}\|(\lambda + \cA)^{-1}\cP\divergence(F)\|_{\L^2} + \|\nabla(\lambda + \cA)^{-1}\cP\divergence(F)\|_{\L^2} \leq C\|F\|_{\L^2}.
    \end{align*}
    The constant $C$ only depends on $\theta$, $d$, $\mu^{\bullet}$ and $\mu_{\bullet}$.
\end{proposition}
Finally, we present an off-diagonal type estimate for the operator family $((\lambda + \cA)^{-1}\cP\divergence)_{\lambda\in \S_\omega}$ which play a key role in the proof of the square root property below.
\begin{proposition}
\label{Prop: off-diagonal type estimates for div}
    There exists $\omega \in (\pi/2,\pi)$ such that for all $\theta \in (0,\omega)$ and $\nu \in (0, d+2)$ there exists $C>0$ such that for all balls $B=B(x_0,r)$, $\lambda \in \S_\theta$ and $F\in \L^2(\IR^d; \IC^{d \times d})$ the following nonlocal off-diagonal estimates are valid
    \begin{align*}
        \int_{B} |(\lambda + \cA)^{-1} \cP \divergence(F)|^2\,\d x \leq \frac{C}{|\lambda|}\sum_{n=0}^\infty\sum_{k=0}^n  \bigg(\prod_{s=0}^{n-k-1} \frac{C}{|\lambda|r^22^{2s}}\bigg)2^{-\nu k}\int_{2^{n+1}B} |F|^2_{\IC^{d \times d}} \,\d x.
    \end{align*}
    The constant $C$ only depends on $\theta$, $\nu$, $d$, $\mu^{\bullet}$ and $\mu_{\bullet}$.
\end{proposition}
\begin{proof}
    For the sake of clarity, we will write $\lvert \, \bigdot \, \rvert = \lvert \, \bigdot \, \rvert_{\IC^{d \times d}}$ in this proof. \par
    Let $u\in \H^1_\sigma(\IR^d)$ be given by $u \coloneqq (\lambda + \cA)^{-1} \cP \divergence(F)$. Then in~\cite[Thm.~1.2]{caccioppoli}, the following nonlocal Caccioppoli inequality was proven
    \begin{align*}
        &|\lambda|\sum_{k=0}^\infty 2^{-\nu k}\int_{2^kB}|u|^2\,\d x + \sum_{k=0}^\infty 2^{-\nu k}\int_{2^kB}|\nabla u|^2\,\d x \\
        &\quad \leq \sum_{k=0}^\infty 2^{-(\nu + 2) k}\frac{C}{r^2}\int_{2^{k+1}B}|u|^2\,\d x + C\sum_{k=0}^\infty 2^{-\nu k}\int_{2^{k+1}B}|F|^2\,\d x
    \end{align*}
    for some $C >0$ depending only on $\mu_{\bullet},\mu^{\bullet}, \nu$ and dimension $d$. Dropping the series of the gradients, dividing by $|\lambda|$ and estimating $2^{- (\nu + 2) k} \leq 2^{- \nu k}$ turns the estimate into
    \begin{align}
    \label{Eq: Iteration step}
        &\sum_{k=0}^\infty 2^{-\nu k}\int_{2^kB}|u|^2\,\d x \leq \frac{C}{|\lambda|r^2} \sum_{k=0}^\infty 2^{-\nu k} \int_{2^{k+1}B}|u|^2\,\d x + \frac{C}{|\lambda|}\sum_{k=0}^\infty 2^{-\nu k}\int_{2^{k+1}B}|F|^2\,\d x.
    \end{align}
    Observe that on the right-hand side almost the same term as on the left-hand side appears but that $B$ is replaced by $2 B$. Thus, this term can be estimated by using~\eqref{Eq: Iteration step} once again but with $B$ replaced by $2 B$. This leads to
    \begin{align*}
        \sum_{k=0}^\infty 2^{-\nu k}\int_{2^kB}|u|^2\,\d x &\leq \frac{C}{|\lambda|r^2}\frac{C}{|\lambda|(2r)^2} \sum_{k=0}^\infty 2^{-\nu k} \int_{2^{k+2}B}|u|^2\,\d x\\
        &\quad + \frac{C}{|\lambda|}\sum_{\ell=0}^1\sum_{k=0}^\infty \bigg(\prod_{s=0}^{\ell-1} \frac{C}{|\lambda|r^22^{2s}}\bigg)2^{-\nu k}\int_{2^{k+\ell+1}B}|F|^2\,\d x.
    \end{align*}
    Iterating this procedure with growing radii leads in the limit to 
    \begin{align*}
        \sum_{k=0}^\infty 2^{-\nu k}\int_{2^kB}|u|^2\,\d x &\leq   \frac{C}{|\lambda|}\sum_{\ell=0}^\infty \sum_{k=0}^\infty \bigg(\prod_{s=0}^{\ell-1} \frac{C}{|\lambda|r^22^{2s}}\bigg)2^{-\nu k}\int_{2^{k+\ell+1}B}|F|^2\,\d x\\
        &=   \frac{C}{|\lambda|}\sum_{k=0}^\infty\sum_{\ell=0}^\infty  \bigg(\prod_{s=0}^{\ell-1} \frac{C}{|\lambda|r^22^{2s}}\bigg)2^{-\nu k}\int_{2^{k+\ell+1}B}|F|^2\,\d x\\
        &=   \frac{C}{|\lambda|}\sum_{k=0}^\infty\sum_{n=k}^\infty  \bigg(\prod_{s=0}^{n-k-1} \frac{C}{|\lambda|r^22^{2s}}\bigg)2^{-\nu k}\int_{2^{n+1}B}|F|^2\,\d x\\
        &=   \frac{C}{|\lambda|}\sum_{n=0}^\infty\sum_{k=0}^n  \bigg(\prod_{s=0}^{n-k-1} \frac{C}{|\lambda|r^22^{2s}}\bigg)2^{-\nu k}\int_{2^{n+1}B}|F|^2\,\d x. \qedhere
    \end{align*}
\end{proof}
\begin{corollary}
\label{Cor: theta OD}
For all $\nu \in (0, d+2)$ there exists a constant $C>0$ such that for all $t > 0$, balls $B=B(x_0,r)$ with $r \geq t$ and $F\in \L^2(\IR^d; \IC^{d \times d})$ we have
\begin{align*}
 \int_{B}|t(1+t^2 \cA)^{-1}\cP\divergence(F)|^2\,\d x &\leq C\sum_{n=0}^\infty  2^{-\nu n}\int_{C_n(B)}|F|_{\IC^{d \times d}}^2\, \d x.
\end{align*}
The constant $C$ only depends on $\nu$, $d$, $\mu^{\bullet}$ and $\mu_{\bullet}$.
\end{corollary}

\begin{proof}
For the sake of clarity, we will write $\lvert \, \bigdot \, \rvert = \lvert \, \bigdot \, \rvert_{\IC^{d \times d}}$ in this proof. \par
First observe that $\lambda^{\frac{1}{2}} (\lambda + \cA)^{-1} \cP \divergence(F) = t(1+t^2 \cA)^{-1}\cP\divergence(F)$ if $\lambda = t^{-2}$. Next, an application of Proposition~\ref{Prop: off-diagonal type estimates for div} with $\nu^{\prime} \in (\nu , d + 2)$ yields
    \begin{align*}
        \int_{B}|t(1+t^2 \cA)^{-1}\cP\divergence(F)|^2\,\d x &\lesssim\sum_{n=0}^\infty\sum_{k=0}^n  \bigg(\prod_{s=0}^{n-k-1} \frac{C}{\frac{r^2}{t^2}2^{2s}} \bigg) 2^{-\nu^{\prime} k} \int_{2^{n+1}B}|F|^2\,\d x\\
        &\leq \sum_{n=0}^\infty\sum_{k=0}^n C^{n-k} 2^{-(n-k-1)(n-k)} 2^{-\nu^{\prime} k} \int_{2^{n+1}B}|F|^2\,\d x.
    \end{align*}
Without loss of generality, assume that $C \geq 1$. Let $k_0 \in \IN_0$ be such that $C \leq 2^{k_0 - 1 - \nu^{\prime}}$, so that $C 2^{- (n - k - 1)} \leq 2^{- \nu^{\prime}}$ whenever $n - k \geq k_0$. Then
\begin{align*}
 \sum_{k=0}^n C^{n-k} 2^{-(n-k-1)(n-k)} 2^{-\nu^{\prime} k} \leq \sum_{k = 0}^{n - k_0} 2^{- \nu^{\prime}(n - k)} 2^{- \nu^{\prime} k} + \sum_{k = n - k_0 + 1}^n C^{k_0 - 1} 2^{- \nu^{\prime} k} \lesssim 2^{- \nu n}.
\end{align*}
Consequently,
\begin{align*}
 \sum_{n=0}^\infty\sum_{k=0}^n C^{n-k} 2^{-(n-k-1)(n-k)} 2^{-\nu^{\prime} k} \int_{2^{n+1}B}|F|^2\,\d x \lesssim \sum_{n=0}^\infty2^{-\nu n} \int_{2^{n+1}B} |F|^2 \, \d x.
\end{align*}
Splitting $2^{n + 1} B$ into annuli finally yields
\begin{align*}
    \sum_{n=0}^\infty2^{-\nu n}\int_{2^{n+1}B}|F|^2\,\d x = \sum_{n=0}^\infty \sum_{\ell = 0}^n 2^{-\nu n} \int_{C_{\ell}(B)}|F|^2\,\d x &\lesssim\sum_{\ell=0}^\infty2^{-\nu \ell}\int_{C_{\ell}(B)}|F|^2\,\d x. \qedhere
\end{align*}
\end{proof}

\begin{remark}
\label{Rem: Off-diagonals}
\begin{enumerate}
 \item One can replace balls by cubes in the statements of Proposition~\ref{Prop: off-diagonal type estimates for div} and Corollary~\ref{Cor: theta OD} since one can prove the nonlocal Caccioppoli inequality for cubes as well.
 \item If $\supp (F) \subseteq \overline{C_n (B)}$ for some $n \in \IN$, then the estimate in Corollary~\ref{Cor: theta OD} turns into
 \begin{align*}
  \| t(1+t^2 \cA)^{-1}\cP\divergence(F)\|_{\L^2 (B)} &\leq C 2^{-\frac{\nu}{2} n} \|F\|_{\L^2(C_n(B))},
 \end{align*}
 which is an off-diagonal estimate of polynomial order $\frac{\nu}{2}$.
\end{enumerate}
\end{remark}

\section{Reduction to a square function estimate}
\label{Sec: Reduction to a square function estimate}

\noindent As for elliptic operators, the square root property for the generalized Stokes operator is equivalent to the square function estimates
\begin{align}
\label{eq:SFE start solution Kato}
    \int_0^\infty \|t A (1+t^2A)^{-1} u\|_{\L^2}^2 \, \frac{\d t}{t} \lesssim \|\nabla u\|_{\L^2}^2 \qquad (u \in \H^1_\sigma(\IR^d))
\end{align}
and
\begin{align*}
    \int_0^\infty \|t A^* (1+t^2A^*)^{-1} u\|_{\L^2}^2 \, \frac{\d t}{t} \lesssim \|\nabla u\|_{\L^2}^2 \qquad (u \in \H^1_\sigma(\IR^d)),
\end{align*}
where $A^*$ is the Hilbert space adjoint to $A$.  An abstract argument for this was presented in~\cite[Prop.~12.7]{ISEM}. We shortly recapitulate the proof of necessity of this equivalence here. \par
Since $A^*$ is a generalized Stokes operator which coefficients (given by $(\mu^*)_{\alpha \beta}^{i j} \coloneqq \overline{\mu_{\beta \alpha}^{j i}}$) still satisfy Assumption~\ref{Ass: Coefficients} it will suffice to establish~\eqref{eq:SFE start solution Kato} for the class of all generalized Stokes operators subject to Assumption~\ref{Ass: Coefficients}. \par
Consider the holomorphic function $f(z) \coloneqq \sqrt{z} (1 + z)^{-1}$ on $\S_\pi$. As $A$ is m-accretive it has a bounded $\H^{\infty}$-calculus due to von Neumann's inequality~\cite[Thm.~7.1.7]{Haase}. McIntosh's theorem~\cite[Thm.~7.3.1]{Haase} implies that $A$ satisfies quadratic estimates, i.e., that
\begin{align}
\label{Eq: Quadratic estimate}
 \int_0^{\infty} \| f (t^2A) v \|^2_{\L^2} \, \frac{\d t}{t} \simeq \| v \|_{\L^2}^2 \qquad (v \in \L^2_\sigma(\IR^d)).
\end{align}
Given $u \in \dom(A)$, one calculates $f(t^2 A) A^{1/2} u = t A (1 + t^2 A)^{-1} u$, so that~\eqref{eq:SFE start solution Kato} combined with~\eqref{Eq: Quadratic estimate} applied to $v = A^{1/2} u$ turns into
\begin{align*}
 \| A^{1/2} u \|_{\L^2}^2 \lesssim \int_0^\infty \|t A (1+t^2A)^{-1} u\|_{\L^2}^2 \, \frac{\d t}{t} \lesssim \| \nabla u \|_{\L^2}^2 \qquad (u \in \dom(A)).
\end{align*}
Since $\dom(A)$ is dense in $\H^1_{\sigma} (\IR^d)$ with respect to the $\H^1_{\sigma} (\IR^d)$-norm (this follows abstractly by~\cite[Thm.~VI.2.1]{Kato_perturbation}), the closedness of $A^{1/2}$ yields that $\H^1_{\sigma} (\IR^d) \subseteq \dom(A^{1/2})$ and that
\begin{align*}
 \| A^{1/2} u \|_{\L^2} \lesssim \| \nabla u \|_{\L^2} \qquad (u \in \H^1_{\sigma} (\IR^d)).
\end{align*}
Replacing $A$ by $A^*$ we deduce the same estimate for $A^*$. Thus, Assumption~\ref{Ass: Coefficients} yields for $u \in \dom(A)$  
\begin{align*}
 \mu_{\bullet} \| \nabla u \|_{\L^2}^2 \leq \Re \fa(u , u) = \Re \langle A u , u \rangle_{\L^2} \leq \| A^{1/2} u \|_{\L^2} \| (A^*)^{1/2} u \|_{\L^2} \lesssim \| A^{1/2} u \|_{\L^2} \| \nabla u \|_{\L^2}.
\end{align*}
Since $\dom(A)$ is a core for $A^{1/2}$ (see~\cite[Thm.~V.3.35]{Kato_perturbation}), we conclude that $\dom(A^{1/2}) \subseteq \H^1_{\sigma} (\IR^d)$ and that
\begin{align*}
 \| \nabla u \|_{\L^2} \lesssim \| A^{1/2} u \|_{\L^2} \qquad (u \in \dom(A^{1/2}))
\end{align*}
by density. Thus, Theorem~\ref{Thm: Kato} is proved once the square function estimate~\eqref{eq:SFE start solution Kato} can be established for all generalized Stokes operators $A$ which coefficients satisfy Assumption~\ref{Ass: Coefficients}.

\section{Principle part approximation}
\label{Sec: Principle part approximation}

\noindent Since the square function estimate~\eqref{eq:SFE start solution Kato} requires a control by $F\coloneqq \nabla u$ we rewrite the expression on the left-hand side in terms of $\nabla u$ by using Lemma~\ref{Lem: Product structure}
\begin{align}
\label{eq:introducing theta}
    tA (1+t^2A)^{-1} u
    = t (1+t^2\cA)^{-1} \cA u = - t (1 + t^2 \cA)^{-1} \cP \divergence (\mu \nabla u)
    \eqqcolon \Theta_t (\nabla u).
\end{align}
\begin{definition}
\label{defi:theta}\index{$\Theta_t$ (proof the Kato conjecture)}
For $t>0$ define the bounded operator $\Theta_t$ by
\begin{align*}
     \Theta_t \colon \L^2(\IR^d;\IC^{d \times d}) \to \L^2(\IR^d;\IC^d), \quad  \Theta_t F \coloneqq  -t(1+t^2 \cA)^{-1}\cP\divergence (\mu F).
\end{align*}
\end{definition}
Recall that $- t (1 + t^2 \cA)^{-1} \cP \divergence$ defines a bounded family in $\cL(\L^2(\IR^d ; \IC^{d \times d}),\L^2(\IR^d;\IC^d))$ and further satisfies the nonlocal off-diagonal bounds from Proposition~\ref{Prop: off-diagonal type estimates for div} and Corollary~\ref{Cor: theta OD}. Since $\Theta_t$ arises from that family by multiplication by the bounded, matrix-valued function $\mu$, these estimates are transferred to $\Theta_t$ and will be used frequently in the following. A first application is the following definition of $\Theta_t$ on bounded matrix-valued functions.
\begin{proposition}
\label{prop: operator on Linfty}
    Let $b\in \L^\infty(\IR^d;\IC^{d\times d})$. Then for every $t>0$ and ball $B=B(x,r)\subseteq\IR^d$ the limit
    \begin{align*}
        \Theta_t b \coloneqq \lim_{j\to \infty}\Theta_t(\ind_{2^jB}b)
    \end{align*}
    exists in $\L_\loc^2(\IR^d;\IC^d)$ and is independent of $x$ and $r$. Additionally, one obtains the same limit if $B$ is replaced by a cube with sides parallel to the coordinate axes.
\end{proposition}
\begin{proof}
    Let $K\subseteq \IR^d$ be an arbitrary compact set. Then there exists $j_0\in\IN$ such that $K\subseteq 2^{j_0}B$. Assume from now on that $\ell>j>j_0$. Proposition~\ref{Prop: off-diagonal type estimates for div} implies
    \begin{align*}
        &\| \Theta_t(\ind_{2^{\ell}B}b) - \Theta_t(\ind_{2^jB}b)\|^2_{\L^2(K)} \leq \|\Theta_t(\ind_{2^{\ell}B\setminus 2^jB}b)\|^2_{\L^2(2^{j_0}B)}\\
        &\qquad \lesssim \sum_{n=0}^\infty\sum_{k=0}^n  \bigg(\prod_{s=0}^{n-k-1}\frac{C}{\frac{2^{2(s+j_0)}r^2}{t^2}}\bigg)2^{-\nu k} \cdot \|\ind_{2^{\ell}B\setminus 2^jB}b\|^2_{\L^2(2^{j_0+n+1}B)}\\
        &\qquad \leq \sum_{n=j-j_0}^\infty\sum_{k=0}^n  \bigg(\prod_{s=0}^{n-k-1}\frac{C}{\frac{2^{2(s+j_0)}r^2}{t^2}}\bigg)2^{-\nu k} \cdot \|b\|^2_{\L^2(2^{j_0+n+1}B)}\\
        &\qquad \leq \sum_{n=j-j_0}^\infty\sum_{k=0}^n  \bigg(\prod_{s=0}^{n-k-1}\frac{C}{\frac{2^{2(s+j_0)}r^2}{t^2}}\bigg)2^{-\nu k} \cdot |2^{j_0+n+1}B|\, \|b\|^2_{\L^\infty} \\
        &\qquad \leq \sum_{n=j-j_0}^\infty\sum_{k=0}^n \bigg( 
        \frac{C t^2}{2^{2 j_0} r^2} \bigg)^{n - k} 2^{- (n - k - 1) (n - k)} 2^{-\nu k} \cdot |2^{j_0+n+1}B|\, \|b\|^2_{\L^\infty}.
    \end{align*}
Choosing $j_0$ large enough, we can assume that $C t^2 / (2^{2 j_0} r^2) \leq 1$. Moreover, maximizing the coefficient $2^{- (n - k - 1) (n - k)} 2^{-\nu k}$ with respect to $k$ shows that
    \begin{align*}
        \| \Theta_t(\ind_{2^{\ell}B}b) - \Theta_t(\ind_{2^jB}b)\|^2_{\L^2(K)} &\leq C(\nu)\sum_{n=j-j_0}^\infty n 2^{-\nu n} \cdot |2^{j_0+n+1}B|\,\|b\|^2_{\L^\infty} \\
        &\leq C(\nu,d,j_0,r)\sum_{n=j-j_0}^\infty n 2^{(d-\nu) n}\|b\|^2_{\L^\infty},
    \end{align*}
    which converges to zero for $\ell,j\to \infty$ because we can choose $\nu>d$. Hence, $(\Theta_t(\ind_{2^jB}b))_{j\in\IN}$ is a Cauchy sequence in $\L^2(K)$ for every $K$ which implies the claim. $B$-independence of the limit as well as that $B$ could be replaced by a cube follows by an analogous pattern, see, e.\@g.\@,~\cite[Prop.~5.1]{Auscher_Egert}. Recall that (as was mentioned in Remark~\ref{Rem: Off-diagonals}) the off-diagonal estimates hold for cubes as well.
\end{proof}
\begin{definition}[Principal part of $\Theta_t$]
\label{defi:gamma}
Identify the matrices $e_{jk} = (\delta_{jm}\delta_{kn})_{m,n=1}^{d}\in \IC^{d\times d}$, where $\delta_{mn}$ denotes Kronecker's delta, with the respective constant functions on $\IR^d$. For $t>0$ we define the \emph{principle part of $\Theta_t$} as
\begin{align*}
    (\gamma_t)_{jk} \coloneqq  \Theta_t (e_{jk})\in \L_\loc^2(\IR^d;\IC^{d})
\end{align*}
for all $j,k=1, \dots , d$. In many situations, we will regard $\gamma_t (x)$, for $x \in \IR^d$ and $t > 0$ fixed, as a linear operator in $\cL(\IC^{d \times d} , \IC^d)$ defined via
\begin{align*}
    \gamma_t (x) \cdot G 
    \coloneqq \sum_{j,k=1}^d(\gamma_t (x))_{jk} G_{jk} \qquad (G \in \IC^{d \times d}).
\end{align*}
The norm of $\gamma_t (x) \cdot$ will be denoted by $\lvert \gamma_t (x) \cdot \rvert_{\cL(\IC^{d \times d} , \IC^d)}$.
\end{definition}
The idea of the principle part is to approximate the operator $\Theta_t$ in an averaged sense in order to reduce the square function estimate~\eqref{eq:SFE start solution Kato} to a certain Carleson measure estimate. For this purpose, we next introduce the dyadic averaging operator as in~\cite[Sec.~9.3]{ISEM}.
\begin{definition}
\begin{enumerate}
    \item For each $j\in\IZ$ we define \emph{dyadic cubes of generation $2^j$} as elements of the set
    \begin{align*}
        \cube_{2^j}\coloneqq \{ 2^jx + [0,2^j)^d: x\in\IZ^d\}
    \end{align*}
    and denote by $\cube \coloneqq \cup_{j\in\IZ} \cube_{2^j}$ the collection of all dyadic cubes. Moreover, for $t>0$ set $\cube_t \coloneqq \cube_{2^j}$ for the unique integer with $2^{j-1}<t\leq 2^j$.

    \item For $u\in\Lloc^1(\IR^d)$ and $t>0$ define the \emph{dyadic averaging operator (at scale $t$)} as
    \begin{align*}
        (\DyadicMax_tu) (x) \coloneqq \fint_Q u(y)\,\d y
    \end{align*}
    where $Q\in\cube_t$ is the unique dyadic cube which contains $x\in\IR^d$.
\end{enumerate}
\end{definition}

When a scalar operator, such as the averaging operator, is applied to a vector- or matrix-valued function, we agree from now on to apply the operator to each component separately. 

In the following proposition, the principle part will be applied to an average of $\nabla u$. Observe that the trace of the matrix $\nabla u$ satisfies $\tr(\nabla u) = 0$ because $u$ is divergence-free. To capture this fact, we introduce the following projection onto trace-free matrices on $\IC^{d \times d}$
\begin{align*}
 \P : \IC^{d \times d} \to \IC^{d \times d}, \quad \P G := G - \frac{\tr(G)}{d} \Id_{d \times d}
\end{align*} 
where $\Id_{d \times d}\in \IC^{d \times d}$ is the identity matrix. The map $\gamma_t (x) \cdot \P \in \Lop(\IC^{d \times d} , \IC^d)$ will be understood as $\gamma_t (x) \cdot \P G \coloneqq \gamma_t (x) \cdot (\P G)$ for $G \in \IC^{d \times d}$.

\begin{proposition}[Reduction to a Carleson measure estimate]
\label{prop:reduction to Carleson}
    If
    \begin{align}
    \label{eq: principle part approximation}
        \int_0^\infty \|(\Theta_t - \gamma_t \cdot \DyadicMax_t)(\nabla u)\|_{\L^2}^2 \, \frac{\d t}{t} \lesssim \|\nabla u\|_{\L^2}^2 \qquad (u \in \H^1_\sigma(\IR^d)),
    \end{align}
    and if $|\gamma_t
    (x) \cdot \P|_{\cL(\IC^{d \times d} , \IC^d)}^2 \, \frac{\d x \, \d t}{t}$ is a Carleson measure, then~\eqref{eq:SFE start solution Kato} is valid for all $u \in \H^1_{\sigma} (\IR^d)$.
\end{proposition}
\begin{proof}
    Observe that $\gamma_t (x) \cdot (\DyadicMax_t \nabla u) (x) = \gamma_t (x) \cdot \P (\DyadicMax_t \nabla u) (x)$ since $\divergence(u) = 0$. Using~\eqref{eq: principle part approximation} we estimate
    \begin{align*}
        \int_0^\infty \|\Theta_t(\nabla u)\|_{\L^2}^2 \, \frac{\d t}{t} &\lesssim \int_0^\infty \|(\Theta_t - \gamma_t \cdot \DyadicMax_t)(\nabla u)\|_{\L^2}^2 \, \frac{\d t}{t} \\
        &\qquad + \int_0^\infty \int_{\IR^d}  |(\DyadicMax_t\nabla u)(x)|_{\IC^{d \times d}}^2 |\gamma_t(x) \cdot \P|_{\cL(\IC^{d \times d} , \IC^d)}^2 \, \frac{\d x \, \d t}{t}\\
        &\lesssim \|\nabla u\|_{\L^2}^2 + \int_0^\infty \int_{\IR^d} |(\DyadicMax_t\nabla u)(x)|_{\IC^{d \times d}}^2 |\gamma_t(x) \cdot \P|_{\cL(\IC^{d \times d} , \IC^d)}^2 \, \frac{\d x \, \d t}{t}.
    \end{align*}
    Under the assumption that $|\gamma_t(x) \cdot \P|_{\cL(\IC^{d \times d} , \IC^d)}^2 \, \frac{\d x \, \d t}{t}$ is a Carleson measure, the second term can be estimated by the $\L^2$-norm of $\nabla u$ due to Carleson's lemma, see~\cite[Thm.~9.19]{ISEM}. The relation~\eqref{eq:introducing theta} now implies the claim.
\end{proof}
\noindent The rest of this section is devoted to establish~\eqref{eq: principle part approximation}. We follow the approach of~\cite{ISEM} by first showing uniform bounds of $\gamma_t \cdot \DyadicMax_t$ and its approximation $\Theta_t - \gamma_t \cdot \DyadicMax_t$ in $\L^2(\IR^d;\IC^d)$. The following representation of  $\gamma_t \cdot \DyadicMax_tF$ for $F\in \L^2(\IR^d;\IC^{d\times d})$ will play a useful role in their proofs. For a fixed dyadic cube $Q \in \cube_t$ we have for $x \in Q$
\begin{align*}
     \gamma_t (x) \cdot \DyadicMax_t F (x)
    = \sum_{j,k=1}^d [\Theta_t(e_{jk})](x) \, (F_{jk})_Q
    = \bigg[\Theta_t \biggl(\sum_{j,k=1}^d e_{jk} \, (F_{jk})_Q \biggr)\bigg]  (x)
    = [\Theta_t \bigl(\ind_{\IR^d} (F)_Q \bigr)](x).
\end{align*}
Thus, by virtue of Proposition~\ref{prop: operator on Linfty}, we obtain
\begin{align}
\label{eq:formula gamma_t A_t on a cube}
\begin{split}
   \gamma_t \cdot \DyadicMax_t F = \lim_{j \to \infty} \Theta_t \bigl(\ind_{2^jQ}(F)_Q \bigr) = \lim_{j \to \infty} \sum_{\ell=0}^{j-1}\Theta_t \bigl(\ind_{C_\ell(Q)}(F)_Q \bigr) =  \sum_{\ell=0}^{\infty}\Theta_t \bigl(\ind_{C_\ell(Q)}(F)_Q \bigr).
\end{split}
\end{align}
\begin{lemma}
\label{lem:uniform bound for principal part averages}
For all $t>0$ we have
\begin{align*}
    \|\gamma_t \cdot \DyadicMax_t F\|_{\L^2} \lesssim \|F\|_{\L^2} \qquad (F \in \L^2(\IR^d;\IC^{d\times d})).
\end{align*}
\end{lemma}
\begin{proof}
It is enough to show
\begin{align}
\label{eq1:uniform bounds for principal part averages}
    \|\gamma_t \cdot \DyadicMax_t F \|^2_{\L^2(Q)} \lesssim \|F\|^2_{\L^2(Q)}
\end{align}
for all cubes $Q \in \cube_t$. As for~\eqref{eq1:uniform bounds for principal part averages}, we use~\eqref{eq:formula gamma_t A_t on a cube} and apply the off-diagonal type estimates for $\Theta_t$ from Corollary~\ref{Cor: theta OD} with $\ell(Q)\geq t$ in order to estimate
\begin{align}
\label{eq2:uniform bound for principal part averages}
\begin{split}
    \|\gamma_t \cdot \DyadicMax_t F \|_{\L^2(Q)}
    &\leq  \sum_{\ell=0}^{\infty}\bigl\|\Theta_t \bigl(\ind_{C_\ell(Q)}(F)_Q \bigr)\bigr\|_{\L^2(Q)}\\
    &\leq C\sum_{\ell=0}^{\infty} \bigg( \sum_{n=0}^\infty  2^{-\nu n}\|\ind_{C_\ell(Q)}(F)_Q\|^2_{\L^2(C_n(Q))} \bigg)^{\frac{1}{2}} \\
    &= C\sum_{\ell=0}^{\infty}  2^{-\frac{\nu}{2} \ell}\|(F)_Q\|_{\L^2(C_\ell(Q))}.
\end{split}
\end{align}
By Hölder's inequality, it follows
\begin{align*}
    \|(F)_Q\|_{\L^2(C_\ell(Q))} =  \frac{|C_\ell(Q)|^\frac{1}{2}}{|Q|}\bigg|\int_Q F\,\d x\bigg|_{\IC^{d \times d}}
    \leq   \frac{|C_\ell(Q)|^\frac{1}{2}}{|Q|^\frac{1}{2}}\| F\|_{\L^2(Q)}
    =   2^{\frac{d}{2}\ell}\| F\|_{\L^2(Q)}.
\end{align*}
Hence, we get
\begin{align*}
     \|\gamma_t \cdot \DyadicMax_t F \|_{\L^2(Q)} &\leq  C\sum_{\ell=0}^{\infty}  2^{\frac{d-\nu}{2} \ell} \|F\|_{\L^2(Q)} \leq  C \|F\|_{\L^2(Q)}
\end{align*}
where we used in the last inequality that the choice $\nu>d$ is admissible. This proves~\eqref{eq1:uniform bounds for principal part averages}. 
\end{proof}

The following lemma provides already an estimate in the direction of~\eqref{eq: principle part approximation} but with a fractional Sobolev norm instead of an $\L^2$-norm on the right-hand side. Observe, that the information that $F$ is a gradient is not (yet) used in this estimate. Note that such a control by a fractional Sobolev norm was also used in the resolution of the parabolic Kato problem, see~\cite[Sec.~7.3]{Auscher_Egert_Nystroem}.

\begin{lemma}
\label{lem:uniform bound principal part approximation}
For all $t>0$ and $\alpha \in (0 , 1)$ we have
\begin{align*}
    \|(\Theta_t - \gamma_t \cdot \DyadicMax_t)F\|_{\L^2} \lesssim t^\alpha \|F\|_{\dot\H^\alpha} \qquad (F \in \H^\alpha(\IR^d;\IC^{d\times d})).
\end{align*}
\end{lemma}
\begin{proof}
Similar to the proof in~\cite[Lem.~13.6]{ISEM} we argue in two steps: First we work on a fixed dyadic cube and then sum over a partition of $\IR^d$. So fix a cube $Q \in \cube_t$. Since $\Theta_t$ is bounded from $\L^2(\IR^d;\IC^{d\times d})$ to $\L^2_{\sigma} (\IR^d)$, we have
\begin{align*}
    \Theta_t F =  \sum_{\ell=0}^{\infty}\Theta_t \bigl(\ind_{C_\ell(Q)}F\bigr) \qquad (F \in \L^2 (\IR^d ; \IC^{d \times d}))
\end{align*}
with convergence in $\L^2 (\IR^d ; \IC^d)$. Subtracting~\eqref{eq:formula gamma_t A_t on a cube} on $Q$ and applying the off-diagonal estimates from Corollary~\ref{Cor: theta OD}, we find
\begin{align*}
    \|(\Theta_t - \gamma_t \cdot \DyadicMax_t)F\|_{\L^2(Q)} 
    & \leq  \sum_{\ell=0}^{\infty} \bigl\|\Theta_t \bigl(\ind_{C_\ell(Q)}(F - (F)_Q) \bigr)\bigr\|_{\L^2(Q)} \\
    &\lesssim    \sum_{\ell=0}^{\infty} \bigg(\sum_{n=0}^{\infty} 2^{-\nu n} \bigl\|\ind_{C_\ell(Q)}(F - (F)_Q) \bigr\|^2_{\L^2(C_n(Q))} \bigg)^{\frac{1}{2}}\\
    &=  \sum_{\ell=0}^{\infty}2^{-\frac{\nu}{2} \ell} \|F - (F)_Q \|_{\L^2(2^{\ell+1}Q)}.
\end{align*}
Next, the aim is to control the right-hand side by a fractional Sobolev--Poincar\'e inequality. For a sufficiently good control on the dependency of the implicit constant on $\ell$, we introduce a telescopic sum and estimate
\begin{align*}
    &\|F - (F)_Q \|_{\L^2(2^{\ell+1}Q)}\\
    &\quad \leq \|F - (F)_{2^{\ell+1}Q} \|_{\L^2(2^{\ell+1}Q)} + \sum_{j=1}^{\ell+1}\|(F)_{2^{j}Q}  - (F)_{2^{j-1}Q} \|_{\L^2(2^{\ell+1}Q)}\\
    &\quad = \|F - (F)_{2^{\ell+1}Q} \|_{\L^2(2^{\ell+1}Q)} + |2^{\ell+1}Q|^\frac{1}{2}\sum_{j=1}^{\ell+1}|(F)_{2^{j}Q}  - (F)_{2^{j-1}Q}|_{\IC^{d \times d}}\\
    &\quad=  \|F - (F)_{2^{\ell+1}Q} \|_{\L^2(2^{\ell+1}Q)} + |2^{\ell+1}Q|^\frac{1}{2}\sum_{j=1}^{\ell+1}\frac{1}{|2^{j-1}Q|}\bigg| \int_{2^{j-1}Q} F- (F)_{2^jQ}\,\d x\bigg|_{\IC^{d \times d}}\\
    &\quad \leq \|F - (F)_{2^{\ell+1}Q} \|_{\L^2(2^{\ell+1}Q)} + |2^{\ell+1}Q|^\frac{1}{2}\sum_{j=1}^{\ell+1}\frac{1}{|2^{j-1}Q|^\frac{1}{2}} \|F- (F)_{2^jQ}\|_{\L^2(2^{j}Q)}
\end{align*}
where we used Hölder's inequality in the last step. Now, apply the fractional Sobolev--Poincar\'e inequality from~\cite[Lem.~2.2]{Fractional_poincare} componentwise with $p,q=2$ and $\alpha\in (0,1)$: There is a constant $C$, depending only on $d$, such that
\begin{align}
\label{eq:stupid application of Poincare on large cubes}
\begin{split}
    \|F - (F)_{2^{j}Q} \|_{\L^2(2^{j}Q)}
    \leq C|2^jQ|^\frac{\alpha}{d} \|F\|_{\dot\H^\alpha(2^jQ)} 
    \leq C t^\alpha 2^{\alpha j} \|F\|_{\dot\H^\alpha(2^jQ)}
\end{split}
\end{align}
for all $j\in \IN$. This leads to
\begin{align*}
    &\|F - (F)_Q \|_{\L^2(2^{\ell+1}Q)}\\
    &\quad \lesssim t^\alpha 2^{\alpha (\ell+1)} \|F\|_{\dot\H^\alpha(2^{\ell+1}Q)} + |2^{\ell+1}Q|^\frac{1}{2}\sum_{j=1}^{\ell+1}\frac{t^\alpha 2^{\alpha j} }{|2^{j-1}Q|^\frac{1}{2}} \|F\|_{\dot\H^\alpha(2^jQ)}\\
    &\quad \simeq t^\alpha 2^{\alpha \ell} \|F\|_{\dot\H^\alpha(2^{\ell+1}Q)} + t^\alpha 2^{\frac{d}{2} \ell}\sum_{j=1}^{\ell+1} 2^{\alpha j-\frac{d}{2}j} \|F\|_{\dot\H^\alpha(2^jQ)}.
\end{align*}
Summarizing, we showed for $0 < \nu^{\prime} < \nu$
\begin{align*}
    &\|(\Theta_t - \gamma_t \cdot \DyadicMax_t)F\|^2_{\L^2(Q)}\\
    &\quad \lesssim t^{2\alpha}\sum_{\ell=0}^{\infty}2^{-\nu^{\prime} \ell}\biggl( 2^{\alpha \ell} \|F\|_{\dot\H^\alpha(2^{\ell+1}Q)} +  2^{\frac{d}{2}\ell}\sum_{j=1}^{\ell+1} 2^{\alpha j-\frac{d}{2} j} \|F\|_{\dot\H^\alpha(2^jQ)} \biggr)^2\\
    &\quad \lesssim t^{2\alpha} \sum_{\ell=0}^{\infty}2^{-\nu^{\prime} \ell}\biggl( 2^{2\alpha \ell} \|F\|^2_{\dot\H^\alpha(2^{\ell+1}Q)} +  (\ell+1)2^{d \ell}\sum_{j=1}^{\ell+1} 2^{2\alpha j-d j} \|F\|^2_{\dot\H^\alpha(2^jQ)} \biggr).
\end{align*}
Now, summing over all $Q \in \cube_t$ yields
\begin{align*}
    &\|(\Theta_t - \gamma_t \cdot \DyadicMax_t)F\|_{\L^2}^2\\
    &\lesssim t^{2\alpha} \sum_{Q \in \smallcube_t}  \sum_{\ell=0}^{\infty}2^{-\nu^{\prime} \ell}\biggl( 2^{2\alpha \ell} \|F\|^2_{\dot\H^\alpha(2^{\ell+1}Q)} +  (\ell+1)2^{d\ell}\sum_{j=1}^{\ell+1} 2^{2\alpha j-d j} \|F\|^2_{\dot\H^\alpha(2^jQ)} \biggr) \\
    &=t^{2\alpha} \sum_{\ell=0}^{\infty}2^{-\nu^{\prime} \ell}\biggl( 2^{2\alpha \ell} \sum_{Q \in \smallcube_t}\|F\|^2_{\dot\H^\alpha(2^{\ell+1}Q)} +  (\ell+1)2^{d \ell}\sum_{j=1}^{\ell+1} 2^{2\alpha j-d j} \sum_{Q \in \smallcube_t}\|F\|^2_{\dot\H^\alpha(2^jQ)} \biggr) .
\end{align*}
Observe that
\begin{align*}
     \sum_{Q \in \smallcube_t}\|F\|^2_{\dot\H^\alpha(2^jQ)} 
     &=  \sum_{Q \in \smallcube_t}\, \int_{2^jQ}\int_{2^jQ} \frac{|F(x)-F(y)|_{\IC^{d \times d}}^2}{|x-y|^{d+2\alpha}}\,\d x \,\d y\\
     &\leq    \int_{\IR^d}\int_{\IR^d} \biggl(\sum_{Q \in \smallcube_t} \ind_{2^jQ}(x)\biggr)\frac{|F(x)-F(y)|_{\IC^{d \times d}}^2}{|x-y|^{d+2\alpha}}\,\d x \, \d y\\
     &\lesssim  2^{dj}  \int_{\IR^d}\int_{\IR^d} \frac{|F(x)-F(y)|_{\IC^{d \times d}}^2}{|x-y|^{d+2\alpha}}\,\d x \, \d y\\
     &=  2^{dj}  \|F\|^2_{\dot\H^\alpha(\IR^d)}
\end{align*}
for all $j\in\IN$. Thus, we conclude
\begin{align*}
    &\|(\Theta_t - \gamma_t \cdot \DyadicMax_t)F\|_{\L^2}^2\\
    &\lesssim t^{2\alpha} \sum_{\ell=0}^{\infty}2^{-\nu^{\prime} \ell}\biggl( 2^{(2\alpha+d) \ell}\|F\|^2_{\dot\H^\alpha(\IR^d)} +  (\ell+1)2^{d \ell}\sum_{j=1}^{\ell+1} 2^{2\alpha j}\|F\|^2_{\dot\H^\alpha(\IR^d)} \biggr)\\
    &\lesssim t^{2\alpha}\|F\|^2_{\dot\H^\alpha(\IR^d)}   \sum_{\ell=0}^{\infty}2^{-\nu^{\prime} \ell}\biggl( 2^{(2\alpha+d) \ell}+  (\ell+1)2^{(2\alpha + d) \ell}\biggr) \\
    &\lesssim t^{2\alpha}\|F\|^2_{\dot\H^\alpha(\IR^d)}
\end{align*}
where the series converges for every $\alpha\in (0,1)$ since we can choose $d+2\alpha <\nu^{\prime} <d+2 $. 
\end{proof}
To prove~\eqref{eq: principle part approximation}, we will first smooth out the ``harsh" approximation to get a control via quadratic estimates of some well-behaved operator and afterwards remove the smoothing again to gain back the full estimate. For this purpose, we introduce the following smoothing operator in the spirit of~\cite{ISEM}.
\begin{definition}
    For $t>0$ define the following \textit{smoothing} operators on $\L^2(\IR^d)$:
    \begin{align*}
        P_t\coloneqq (1-t^2\Delta)^{-1} \quad \text{and} \quad Q_t\coloneqq t\nabla(1-t^2\Delta)^{-1}.
    \end{align*}
\end{definition}
As a reminder, the reader should keep in mind that we agreed to use the same notation for extensions of operators on vector fields by acting componentwise.
\begin{proposition}
\label{prop:smoothened principal part approximation}
    For $F\in \L^2(\IR^d;\IC^{d\times d})$ the following smoothed principle part approximation holds
    \begin{align*}
        \int_0^\infty \|(\Theta_t - \gamma_t \cdot \DyadicMax_t) P_t F\|_{\L^2}^2 \, \frac{\d t}{t} \lesssim\|F\|_{\L^2}^2.
    \end{align*}
\end{proposition}
\begin{proof}
    Due to Lemma~\ref{lem:uniform bound principal part approximation} and the characterization of fractional Sobolev-spaces through the fractional Laplacian (see~\cite[Prop.~3.6]{Hitchhiker}), we have
    \begin{align*}
         \|(\Theta_t - \gamma_t \cdot \DyadicMax_t) P_t F\|_{\L^2} \lesssim t^\alpha\|P_tF\|_{\dot \H^\alpha} = \|t^\alpha(-\Delta)^\frac{\alpha}{2}(1- t^2 \Delta)^{-1}F\|_{\L^2}
    \end{align*}
    for $\alpha\in (0,1)$. We conclude by quadratic estimates of the form~\eqref{Eq: Quadratic estimate} with $f(z) = z^\frac{\alpha}{2}(1+z)^{-1}$ which are valid for the Laplacian by McIntosh's theorem~\cite[Thm.~7.3.1]{Haase}, since the Laplacian admits a bounded $\H^{\infty}$-calculus on $\L^2 (\IR^d)$.
\end{proof}
To remove the smoothing, we split
\begin{align}
\label{eq: ppa splitting}
    (\Theta_t - \gamma_t \cdot \DyadicMax_t) = (\Theta_t - \gamma_t \cdot \DyadicMax_t)P_t + \Theta_t(1-P_t) -  \gamma_t \cdot \DyadicMax_t(1-P_t)
\end{align}
and elaborate square function estimates for the last two terms.
\begin{proposition}
\label{prop:second term estimate}
    For every $u\in \H^1_\sigma(\IR^d)$ we have the square function estimate
    \begin{align*}
        \int_0^\infty \|\Theta_t(1-P_t)\nabla u\|_{\L^2}^2 \, \frac{\d t}{t}\lesssim \|\nabla u\|_{\L^2}^2.
    \end{align*}
\end{proposition}
\begin{proof}
    First observe that for every $u\in \H^1_\sigma(\IR^d)$ we also have $-t^2\Delta (1-t^2 \Delta)^{-1}u \in \H^1_\sigma(\IR^d)$ for all $t>0$ since the Laplacian commutes with derivatives. Thus, we can compute by virtue of~\eqref{eq:introducing theta}
    \begin{align*}
        \Theta_t (1- P_t) \nabla u 
        &=(\Theta_t \nabla) (1-P_t)u \\
        &=\bigl(tA(1+t^2A)^{-1}\bigr) \bigl(-t^2 \Delta (1-t^2 \Delta)^{-1}u\bigr)\\
        &= \bigl(1- (1+t^2A)^{-1}\bigr) \bigl(-t \Delta (1-t^2 \Delta)^{-1}u\bigr) \\
        &= \bigl(1- (1+t^2A)^{-1}\bigr) Q_t^* \nabla u,
    \end{align*}
    where $Q_t^* \nabla u = (Q_t^* \nabla u_1 , \dots , Q_t^* \nabla u_d)^{\top}$. By the sectoriality of $A$, we conclude
    \begin{align*}
        \|\Theta_t (1- P_t) \nabla u\|_{\L^2} \lesssim \|Q_t^* \nabla u\|_{\L^2}
    \end{align*}
    such that the claim follows now from~\cite[Lem.~13.9]{ISEM}.
\end{proof}
\begin{proposition}
\label{prop:third term estimate}
    For all $u\in \H^1_\sigma(\IR^d)$ we have
    \begin{align*}
        \int_0^\infty \|\gamma_t \cdot \DyadicMax_t(1-P_t) \nabla u \|_{\L^2}^2\,\frac{\d t}{t}\lesssim \|\nabla u\|_{\L^2}^2.
    \end{align*}
\end{proposition}
\begin{proof}
    Observe that $\DyadicMax_t = \DyadicMax_t^2$. Then the uniform $\L^2$-bounds of $\gamma_t \cdot \DyadicMax_t$ by Lemma~\ref{lem:uniform bound for principal part averages} reduces the square function estimate to
    \begin{align*}
        \int_0^\infty \|\DyadicMax_t(1-P_t) \nabla u \|_{\L^2}^2\,\frac{\d t}{t}\lesssim \|\nabla u\|_{\L^2}^2 \qquad (u\in \H^1_\sigma(\IR^d)).
    \end{align*}
    This estimate is now independent of the operator $A$ and well-known, see, e.g.,~\cite[Prop.~13.13]{ISEM}.
\end{proof}
The proof of the above proposition uses an interpolation inequality, which will be of use in the last section. Let us state it at this point and refer to~\cite[Lem.~13.12]{ISEM} for a detailed proof.
\begin{lemma}
\label{lem:palmen}
There is a constant $C>0$ such that for all dyadic cubes $Q \in \cube$ and all $u \in \H^1(\IR^d;\IC^d)$ we have
\begin{align*}
    \biggl| \fint_Q \nabla u \, \d x \biggr|_{\IC^{d \times d}}^2 
    \leq \frac{C}{\ell(Q)} \biggl(\fint_Q |u|^2 \, \d x \biggr)^{\frac{1}{2}} \biggl(\fint_Q |\nabla u|_{\IC^{d \times d}}^2 \, \d x \biggr)^{\frac{1}{2}} .
\end{align*}
\end{lemma}
A combination of Propositions~\ref{prop:smoothened principal part approximation},~\ref{prop:second term estimate} and~\ref{prop:third term estimate} with the splitting~\eqref{eq: ppa splitting} establishes the full principal part approximation.

\begin{proposition}[Principal part approximation]
\label{prop:principal part approximation}
We have the square function estimate
\begin{align*}
     \int_0^\infty \|(\Theta_t - \gamma_t \cdot \DyadicMax_t) \nabla u\|_{\L^2}^2 \, \frac{\d t}{t} \lesssim \|\nabla u\|_{\L^2}^2 \qquad (u \in \H^1_\sigma(\IR^d)).
\end{align*}
\end{proposition}


\section{The Carleson measure argument}
\label{Sec: The Carleson measure argument}

\noindent In order to finish the proof of Theorem~\ref{Thm: Kato} it suffices, by virtue of Propositions~\ref{prop:reduction to Carleson} and~\ref{prop:principal part approximation}, to show that
\begin{align*}
 |\gamma_t(x) \cdot \P |_{\cL(\IC^{d \times d} , \IC^d)}^2 \, \frac{\d x \, \d t}{t}
\end{align*}
is a Carleson measure. We shortly recapitulate the notion of a Carleson measure. 

\begin{definition}
A Borel measure $\nu$ on $\IR^{d + 1}_+$ is called a \emph{Carleson measure} if there exists $C > 0$ such that
\begin{align*}
 \nu(R(Q)) \leq C \lvert Q \rvert \qquad (Q \in \cube),
\end{align*}
where $R(Q) \coloneqq Q \times (0 , \ell(Q)]$ denotes a \emph{Carleson box}. The infimum over all constants $C$ is called the \emph{Carleson norm} of $\nu$ and denoted by $\|\nu\|_{\cC}$.
\end{definition}

For elliptic operators in divergence form a key idea in the proof of the Kato property is to use a sectorial decomposition of the range of the principle part. In our situation we have to decompose the range of $\gamma_t \cdot \P$ which is a proper subset of the range of $\gamma_t \cdot$. Since $\gamma_t (x)$ is a $(d\times d)$-matrix with entries in $\IC^d$ we have for $x \in \IR^d$ and $t > 0$
\begin{align*}
 \gamma_t (x) \cdot \P  \subseteq \cV \coloneqq \{ \zeta \cdot \P \colon \zeta \in (\IC^d)^{d \times d} \} \subseteq \cL(\IC^{d \times d} , \IC^d)
\end{align*}
where $\zeta \cdot \P$ is defined in the same way as $\gamma_t (x) \cdot \P$.

\begin{lemma}
\label{Lem: V closed subspace}
$\cV$ is a closed subspace of $\cL(\IC^{d \times d} , \IC^d)$.
\end{lemma}

\begin{proof}
Let $(\xi_n)_{n \in \IN} \subseteq \cV$ converge to some $\xi \in \cL (\IC^{d \times d} , \IC^d)$. Since $\xi_n = \zeta_n \cdot \P$ for some $\zeta_n \in (\IC^d)^{d \times d}$ and $\P$ is a projection we get
\begin{align*}
 \xi_n G = \zeta_n \cdot \P G = \zeta_n \cdot \P^2 G = \xi_n \P G \qquad (G \in \IC^{d \times d}).
\end{align*}
 As a consequence we find that
\begin{align*}
 \xi G = \xi \P G = \sum_{j , k = 1}^{\infty} \xi (e_{j k}) (\P G)_{j k} = \zeta \cdot \P G,
\end{align*}
where $\zeta_{j k} \coloneqq \xi (e_{j k})$ and $e_{j k}$ was defined in Definition~\ref{defi:gamma}.
\end{proof}

The sectorial decomposition will now be defined with respect to the space $\cV$. For the sake of brevity, we will write the $\cL(\IC^{d \times d} , \IC^d)$-norm of elements in $\cV$ as $\lvert \, \bigdot \, \rvert_{\cV}$.

\begin{definition}
\label{defi:cone in Cn}
Let $\eps>0$. For $\xi \in \cV$ with $|\xi|_{\cV} = 1$, define the open cone with central axis $\xi$ by
\begin{align*}
    \Gamma_\xi^\eps \coloneqq \bigl\{ w\in \cV \setminus \{0\} \mid \bigl|\tfrac{w}{|w|_{\cV}} - \xi\bigr|_{\cV} < \eps \bigr\}.
\end{align*}
\end{definition}
Our goal is now to prove the following Carleson measure estimate, where we restrict the values of the principal part to a cone $\Gamma_\xi^\eps$.

\begin{proposition}[Directional Carleson estimate]
\label{prop:key estimate Kato}
There is a choice of $\eps>0$, depending only on $d$, $\mu_{\bullet}$ and $\mu^{\bullet}$, such that for every $\xi \in \cV$ with $\lvert \xi \rvert_{\cV} = 1$ we have
\begin{align*}
    \iint_{R(Q)} |\gamma_t(x) \cdot \P |_{\cV}^2 \, \ind_{\Gamma_\xi^\eps}(\gamma_t(x) \cdot \P) \, \frac{\d x \, \d t}{t} \lesssim |Q| \qquad (Q \in \cube).
\end{align*}
\end{proposition}
Taking the directional Carleson measure estimate from Proposition~\ref{prop:key estimate Kato} for granted, we are ready to prove Theorem~\ref{Thm: Kato}.

\begin{proof}[Proof of Theorem~\ref{Thm: Kato}]
Since $\cV \subseteq \cL(\IC^{d \times d} , \IC^d)$ is closed, by Lemma~\ref{Lem: V closed subspace}, its unit sphere is compact in the subspace topology. Thus, we can cover $\cV \setminus \{ 0 \}$ by finitely many cones $\Gamma_{\xi_1}^\eps, \ldots, \Gamma_{\xi_N}^\eps$, where $N$ depends only on $d$ and $\eps$. The estimate from Proposition~\ref{prop:key estimate Kato} yields
\begin{align*}
    \iint_{R(Q)} |\gamma_t(x) \cdot \P|_{\cV}^2 \, \frac{\d x \, \d t}{t}
    \leq \iint_{R(Q)} \sum_{j=1}^N |\gamma_t(x) \cdot \P |_{\cV}^2 \, \ind_{\Gamma_{\xi_j}^\eps}(\gamma_t(x) \cdot \P) \, \frac{\d x \, \d t}{t}
    \leq CN |Q| 
\end{align*}
for every cube $Q \in \cube$ as required.
\end{proof} 

The remainder of this section is devoted to the proof or the directional Carleson measure estimate which bases on the construction of a $T(b)$-type test function. More precisely, for given $Q \in \cube$ and $\xi \in \cV$ with $\lvert \xi \rvert_{\cV} = 1$, we strive for a construction of functions $b_Q^{\eps}$ such that
\begin{align*}
 \lvert w \rvert_{\cV} \lesssim \big\lvert w \big( (\DyadicMax_t b_Q^{\eps})(x) \big) \big\rvert \qquad (w \in \Gamma_{\xi}^{\eps})
\end{align*}
for a hopefully large set of $(x , t) \in R(Q)$. The test function $b_Q^{\eps}$ should have two properties:
\begin{enumerate}
 \item It should be a gradient of an $\H^1_{\sigma} (\IR^d)$-function in order to make the principle part approximation accessible.
 \item It should point in the direction of the unit matrix $\overline{\xi} \in \IC^{d \times d}$ that realizes the norm of $\xi$.
\end{enumerate}
The first point will be achieved by making an ansatz for $b_Q^{\eps}$ as the gradient of a suitable resolvent of $A$, i.e.,
\begin{align*}
 b_Q^{\eps} \simeq \nabla (1 + \eps^2 \ell (Q)^2 A)^{-1} \Phi.
\end{align*}
Since $(1 + \eps^2 \ell (Q)^2 A)^{-1}$ is an approximation of the identity if $\eps$ is sufficiently small, the second point roughly holds if $\Phi$ is chosen such that $\nabla \Phi = \overline{\xi}$, i.e., $\Phi(x) \coloneqq \overline{\xi}^{\top} (x - x_Q)$, where $x_Q$ is the center of $Q$. A problem is, that $\Phi$ is neither $\L^2$-integrable nor divergence-free. It will suffice to modify $\Phi$ by a multiplication with a bump function to achieve the $\L^2$-integrability. The divergence-freeness should be achieved by a Bogovski\u{\i} correction. However, this only works if $x \mapsto \overline{\xi}^{\top} (x - x_Q)$ is itself divergence-free which is only the case if the trace of the matrix $\overline{\xi}^{\top}$ vanishes. We will show, that there will be now harm if we replace $\overline{\xi}^{\top}$ by $(\P \overline{\xi})^{\top}$. \par
From now on, the map $\xi \in \cV$ with $\lvert \xi \rvert_{\cV} = 1$ is fixed and we write $\Gamma^\eps \coloneqq \Gamma_\xi^\eps$. Let $\overline{\xi} \in \IC^{d \times d}$ be such that
\begin{align*}
 \lvert \overline{\xi} \rvert_{\IC^{d \times d}} = 1 \quad \text{and} \quad \lvert \xi \overline{\xi} \rvert = \lvert \xi \rvert_{\cV}.
\end{align*}
We will fix this choice of $\bar \xi$ in the following. Before we present the construction of $b_Q^{\eps}$ in detail, let us introduce our final tool, the so-called Bogovski\u{\i} operator.

Fix a cube $Q_0 \coloneqq (-1 , 1)^d$ and define $\cC \coloneqq 2 Q_0 \setminus \overline{Q_0}$. The Bogovski\u{\i} operator $\cB_{\cC} \colon \L^2_0 (\cC) \to \H^1_0 (\cC ; \IC^d)$ denotes the solution operator to the divergence equation for functions $f \in \L^2_0 (\cC)$
\begin{align*}
\left\{ \begin{aligned}
 \divergence(u) &= f && \text{in } \cC, \\
 u &= 0 && \text{on } \partial \cC.
\end{aligned} \right.
\end{align*}
Here, $\L^2_0$ denotes the subspace of $\L^2$ of average-free functions on $\cC$. The operator $\cB_{\cC}$ can be constructed as a bounded operator from $\L^2_0 (\cC)$ onto $\H^1_0 (\cC ; \IC^d)$, see, e.g.,~\cite[Sec.~III.3]{Galdi}.
Thus, by construction, $\cB_{\cC}$ satisfies
\begin{align*}
\divergence(\cB_{\cC} f) = f \quad \text{and} \quad \| \cB_{\cC} f \|_{\H^1 (\cC)} \leq C_{Bog} \| f \|_{\L^2 (\cC)} \qquad (f \in \L^2_0 (\cC))
\end{align*}
for some constant depending only on $d$. By rescaling, one can construct a Bogovski\u{\i} operator on $\alpha \cC$ for $\alpha > 0$ as follows: If $f \in \L^2_0 (\alpha \cC)$, then $f_{\alpha} (x) := \alpha f (\alpha x)$ lies in $\L^2_0 (\cC)$. Define
\begin{align*}
 [\cB_{\alpha \cC} f] (x) := [\cB_{\cC} f_{\alpha}] (\alpha^{-1} x) \qquad (f \in \L^2_0 (\alpha \cC) , \, x \in \alpha \cC).
\end{align*}
Clearly, $\cB_{\alpha \cC}$ is bounded from $\L^2_0 (\alpha \cC)$ onto $\H^1_0 (\alpha \cC ; \IC^d)$ and satisfies $\divergence (\cB_{\alpha \cC} f) = f$. Furthermore, we have
\begin{align}
\label{Eq: Homogeneous estimate Bogovskii}
 \| \nabla \cB_{\alpha \cC} f \|_{\L^2 (\alpha \cC)} \leq C_{Bog} \| f \|_{\L^2 (\alpha \cC)} \qquad (f \in \L^2_0(\alpha \cC))
\end{align}
with the same constant $C_{Bog} > 0$ as above. Finally, by translation we might define Bogovski\u{\i} operators for translated annuli as well and these will again satisfy~\eqref{Eq: Homogeneous estimate Bogovskii} with the same constant. \par
Now, we are prepared for the construction of the $T(b)$-type test function.

\begin{proposition}
\label{prop:the Tb test functions}
There exists $\eps_0 \in (0,1]$ such that for all $0 < \eps \leq \eps_0$, all $\xi \in \cV$ with $\lvert \xi \rvert_{\cV} = 1$ and all cubes $Q \in \cube$ there exists $b_Q^\eps \in \L^2 (\IR^d ; \IC^{d \times d})$ with the following properties:
\begin{enumerate}
    \item \label{item1:the Tb test function} $\|b_Q^\eps\|_{\L^2} \lesssim |Q|^{1/2}$,
    \item \label{item2:the Tb test function} $\displaystyle  \Big\lvert \xi \fint_{Q} b_Q^\eps \, \d x \Big\rvert \geq 1$,
    \item \label{item3:the Tb test function} $\displaystyle \iint_{R(Q)} |\gamma_t(x) \cdot \P (\DyadicMax_t b_Q^\eps)(x)|^2 \, \frac{\d x \, \d t}{t} \lesssim \frac{|Q|}{\eps^2} $.
\end{enumerate}
\end{proposition}
\begin{proof}
We fix $Q \in \cube$ and abbreviate $\ell \coloneqq \ell(Q)$. Similarly, we simplify notation by omitting $\eps$ and $Q$ when constructing the test function $b = b_Q^\eps$.

We start by fixing $\eta \in \C_c^\infty(2Q)$ such that 
\begin{align}
\label{Eq: Properties of cut-off}
\eta = 1 \quad \text{in} \quad Q \quad \text{and} \quad \|\eta\|_{\L^\infty} + \ell \|\nabla \eta\|_{\L^\infty} \leq C. 
\end{align}
With $x_Q$ the center of $Q$, we construct a smooth function with compact support, whose gradient is equal to $\overline{\xi}$ on $Q$ as follows. First, note that
\begin{align*}
 \divergence \big( (\P \overline{\xi})^{\top} (x - x_Q) \big) = 0.
\end{align*}
Let $\cB$ be the Bogovski\u{\i} operator in $2 Q^{\circ} \setminus \overline{Q}$ and define
\begin{align}
\label{eq1:the Tb test function}
 \Phi(x) \coloneqq \eta (x) (\P \overline{\xi})^{\top} (x - x_Q) - \cB \big( \nabla \eta \cdot (\P \overline{\xi})^{\top} (\bigdot - x_Q) \big) (x),
\end{align}
where we regard $\cB \big( \nabla \eta \cdot (\P \overline{\xi})^{\top} (\bigdot - x_Q) \big)$ to be extended by zero outside of $2 Q^{\circ} \setminus \overline{Q}$.
Finally, we define the desired test function as
\begin{align}
\label{eq2:the Tb test function}
 b \coloneqq 2 \nabla (1+\eps^2 \ell^2 A)^{-1} \Phi.
\end{align}
Then, we have
\begin{align}
\label{eq3:the Tb test function}
\begin{aligned}
    \frac{1}{2} b- \nabla \Phi
    &= \nabla \bigl((1+\eps^2 \ell^2 A)^{-1} - 1 \bigr) \Phi \\
    &= \nabla \bigl(-(1+\eps^2 \ell^2 \cA)^{-1}\eps^2 \ell^2 \cA \bigr) \Phi \\
    &= \eps^2 \ell^2 \nabla (1+\eps^2 \ell^2 \cA)^{-1} \cP \divergence (\mu \nabla \Phi).
\end{aligned}
\end{align}
Let us prove that we can pick $\eps$ small enough such that $b$ has the stated properties. \par
In order to prove~\eqref{item1:the Tb test function}, we begin by calculating
    \begin{align*}
        \big\lvert\nabla \Phi (x) \big\rvert_{\IC^{d \times d}} 
        &\leq \big\lvert \nabla \eta(x) \otimes(\P \overline{\xi})^{\top} (x-x_Q) \big\rvert_{\IC^{d \times d}} + \big\lvert \eta(x) \P \overline{\xi} \big\rvert_{\IC^{d \times d}} \\
        &\qquad + \big\lvert \nabla \cB \big( \nabla \eta \cdot (\P \overline{\xi})^{\top} (\bigdot - x_Q) \big) (x) \big\rvert_{\IC^{d \times d}}.
    \end{align*}
    Taking $\L^2$-norms, we obtain by~\eqref{Eq: Homogeneous estimate Bogovskii} and~\eqref{Eq: Properties of cut-off} together with the boundedness of $\P$ and the property $\lvert \overline{\xi} \rvert_{\IC^{d \times d}} = 1$
    \begin{align}
    \label{eq4:the Tb test function}
        \|\nabla \Phi\|_{\L^2}^2 \lesssim \frac{\lvert \P \overline{\xi} \rvert_{\IC^{d \times d}}}{\ell^2} \int_{2 Q} \lvert x - x_Q \rvert^2 \, \d x + \lvert Q \rvert \lvert \P \overline{\xi} \rvert_{\IC^{d \times d}} \lesssim \lvert Q \rvert \lvert \P \overline{\xi} \rvert_{\IC^{d \times d}} \lesssim \lvert Q \rvert.
    \end{align}
    Combining~\eqref{eq3:the Tb test function} and Proposition~\ref{Prop: L2 resolvent bounds}, we find
    \begin{align*}
        \| b- 2\nabla \Phi\|_{\L^2}^2 
        \lesssim \| \mu \nabla \Phi\|_{\L^2}^2
        \lesssim |Q|
    \end{align*}
    and together with~\eqref{eq4:the Tb test function} we arrive at~\eqref{item1:the Tb test function}. \par
    We turn to the proof of statement~\eqref{item2:the Tb test function}. As $\nabla \Phi = \P \overline{\xi}$ on $Q$, we can write
    \begin{align}
    \label{eq5:the Tb test function}
        \fint_Q \big( b - 2 \P \overline{\xi} \big) \, \d x
        = \fint_Q \big( b - 2\nabla \Phi \big) \, \d x
        = 2 \eps^2 \ell^2 \fint_Q \nabla u \, \d x,
    \end{align}
    where in the second step we have used~\eqref{eq3:the Tb test function} and $u \coloneqq (1+\eps^2 \ell^2 \cA)^{-1} \cP \divergence(\mu \nabla \Phi)$. The uniform $\L^2$-bounds in Proposition~\ref{Prop: L2 resolvent bounds} in combination with~\eqref{eq4:the Tb test function} yield
    \begin{align*}
        \|u\|_{\L^2} \leq \frac{C}{\eps \ell} |Q|^{1/2} \quad \text{and} \quad
        \|\nabla u\|_{\L^2} \leq \frac{C}{\eps^2 \ell^2} |Q|^{1/2},
    \end{align*}
    so that estimating the average on the right-hand side of~\eqref{eq5:the Tb test function} by means of Lemma~\ref{lem:palmen} leads us to
    \begin{align*}
        \biggl|\fint_Q \big( b - 2 \P \overline{\xi} \big) \, \d x \biggr|_{\IC^{d \times d}}^2
        \leq \frac{C \eps^4 \ell^4}{\ell |Q|} \|u\|_{\L^2} \|\nabla u\|_{\L^2}
        \leq C\eps,
    \end{align*}
    where $C$ varies from step to step. Since $\xi \in \cV$ and $\P$ is a projection, we have $\xi \overline{\xi} = \xi \P \overline{\xi}$ so that
    \begin{align*}
        \xi \fint_Q b \, \d x = 2 \xi \overline{\xi} + \xi \fint_Q \big( b - 2 \P \overline{\xi} \big) \, \d x.
    \end{align*}
    Recall that $\overline{\xi}$ was chosen in such a way, that $\lvert \xi \overline{\xi} \rvert = \lvert \xi \rvert_{\cV} = 1$ which yields
    \begin{align*}
        \bigg\lvert \xi \fint_Q b \, \d x\bigg\rvert 
        \geq 2 - \lvert \xi \rvert_{\cV} \bigg\lvert \fint_Q \big( b - 2 \P \overline{\xi} \big) \, \d x \bigg\rvert_{\IC^{d \times d}} \geq 2 - \sqrt{ C \eps}.
    \end{align*}
    Now,~\eqref{item2:the Tb test function} follows by taking $\eps \leq C^{-1}$. \par
     For the proof of~\eqref{item3:the Tb test function}, we start with the principal part approximation from Proposition~\ref{prop:principal part approximation} and then use~\eqref{item1:the Tb test function} to bound
    \begin{align}
    \label{eq6:the Tb test functio}
    \begin{split}
    \frac{1}{2} \iint_{R(Q)} |\gamma_t(x) \cdot \P (\DyadicMax_t b)(x)|^2 \, \frac{\d x \, \d t}{t} &\leq \int_0^{\ell} \Big( \|(\gamma_t \cdot \DyadicMax_t - \Theta_t) b\|_{\L^2}^2  + \|\Theta_t b\|_{\L^2}^2 \Big) \, \frac{\d t}{t} \\
    &\lesssim \|b\|_{\L^2}^2  + \int_0^{\ell} \|\Theta_t b\|_{\L^2}^2 \, \frac{\d t}{t} \\
    &\lesssim |Q|  +  \int_0^{\ell} \|\Theta_t b\|_{\L^2}^2 \, \frac{\d t}{t}.
    \end{split}
    \end{align}
    Recall from~\eqref{eq2:the Tb test function} that $b$ is by definition the gradient of a function in $\H^1_{\sigma}(\IR^d)$ and so that applying Proposition~\ref{prop:principal part approximation} was allowed. For the last term, we compute $\Theta_t b$ as
    \begin{align*}
    \Theta_t b 
    &= -2t(1+t^2 \cA)^{-1} \cP \divergence \bigl(\mu \nabla (1+\eps^2 \ell^2 A)^{-1} \Phi\bigr) \\
    &= 2t(1+t^2 \cA)^{-1}  \cA (1+\eps^2 \ell^2 \cA)^{-1} \Phi \\
    &= -t \bigl((1+t^2 A)^{-1}\bigr) \bigl((1+\eps^2 \ell^2 \cA)^{-1} \cP \divergence \bigr) 2\mu \nabla \Phi.
    \end{align*}
    We use Proposition~\ref{Prop: L2 resolvent bounds} once again and~\eqref{eq4:the Tb test function} in order to control
    \begin{align*}
        \|\Theta_t b\|_{\L^2}^2
        \lesssim \frac{t^2}{\eps^2 \ell^2} \|2 \mu \nabla \Phi\|_{\L^2}^2
        \lesssim \frac{t^2}{\eps^2 \ell^2} |Q|.
    \end{align*}
    Finally, integration in $t$ yields
    \begin{align*}
        \int_0^{\ell} \|\Theta_t b\|_{\L^2}^2 \, \frac{\d t}{t}
        \lesssim \frac{1}{\eps^2} |Q|,
    \end{align*}
    which we use back in~\eqref{eq6:the Tb test functio} to conclude provided we take $\eps \leq 1$.
\end{proof}

Having a $T(b)$-type test function at our disposal, the Carleson measure property will follow by a combination of a stopping time argument with a John--Nirenberg lemma for Carleson measures. Concerning the first, we refer to~\cite[Lem.~5.11]{AKM} or~\cite[Lem.~14.8]{ISEM} for a proof.

\begin{lemma}
\label{lem:stopping time}
Let $\eps_0 \in (0 , 1]$ be the parameter from Proposition~\ref{prop:the Tb test functions}. There exists $0 < \eps \leq \eps_0$, depending only on $d$, $\mu_{\bullet}$ and $\mu^{\bullet}$, such that for each dyadic cube $Q \in \cube$ there exists a collection of pairwise disjoint dyadic subcubes $(Q_j)_j$ of $Q$ for which the sets
\begin{align}
\label{eq:dyadic sawwtooth}
    E(Q) \coloneqq Q \setminus \bigcup_{j} Q_{j} \quad \text{and} \quad E^*(Q) \coloneqq R(Q) \setminus \bigcup_{j} R(Q_{j})
\end{align}
have the following properties:
\begin{enumerate}
    \item \label{item1:stopping time} $\bigl|E(Q) \bigr| \geq \eta |Q|$, for some $\eta>0$ depending only on $d$, $\mu_{\bullet}$ and $\mu^{\bullet}$,
    \item \label{item2:stopping time} $\displaystyle \bigl|w (\DyadicMax_t b_Q^\eps)(x)\bigr| \geq \frac{1}{2} |w|$, whenever $(x,t) \in E^*(Q)$ and $w \in \Gamma^\eps$.
\end{enumerate}
\end{lemma}

Lemma~\ref{lem:stopping time}~\eqref{item2:stopping time} guarantees that the Borel measure $\nu$ on $\IR^{d + 1}_+$ defined on Borel sets $E \subseteq \IR^{d + 1}_+$ by
\begin{align*}
 \nu (E) \coloneqq \iint_E |\gamma_t(x) \cdot \P |_{\cV}^2 \, \ind_{\Gamma_\xi^\eps}(\gamma_t(x) \cdot \P) \, \frac{\d x \, \d t}{t}
\end{align*}
satisfies for cubes $Q \in \cube$
\begin{align*}
 \nu (E^* (Q)) = \iint_{E^* (Q)} |\gamma_t(x) \cdot \P |_{\cV}^2 \, \ind_{\Gamma_\xi^\eps}(\gamma_t(x)\cdot \P) \, \frac{\d x \, \d t}{t} \leq 4 \iint_{E^*(Q)} |\gamma_t(x) \cdot \P (\DyadicMax_t b_Q^\eps)(x)|^2 \, \frac{\d x \, \d t}{t}.
\end{align*}
The right-hand side can be controlled by virtue of Proposition~\ref{prop:the Tb test functions}~\eqref{item3:the Tb test function} leading to the estimate
\begin{align}
\label{Eq: Almost Carleson bound}
 \nu (E^*(Q)) \leq \frac{4 C}{\eps^2} \lvert Q \rvert \qquad (Q \in \cube),
\end{align}
where $C$ is the constant from the very proposition. The Carleson measure property of $\nu$ and thus the statement of Proposition~\ref{prop:key estimate Kato} finally follows by a combination of~\eqref{Eq: Almost Carleson bound} with Lemma~\ref{lem:stopping time}~\eqref{item1:stopping time} and the following John--Nirenberg lemma for Carleson measures. We refer to~\cite[Lem.~14.10]{ISEM} for its proof. The argument is nowadays standard in the verification of the square root property of elliptic operators. 

\begin{lemma}
Let $\nu$ be a Borel measure on $\IR^{d + 1}_+$ and suppose that there exist constants $\kappa, \eta > 0$ with the following properties. For every dyadic cube $Q \in \cube$ there exist pairwise disjoint dyadic subcubes $Q_j$ of $Q$ such that the sets $E(Q)$ and $E^*(Q)$ defined in~\eqref{eq:dyadic sawwtooth} satisfy
\begin{enumerate}
    \item $\bigl|E(Q) \bigr| \geq  \eta |Q|$,
    \item $\nu\bigl(E^*(Q)\bigr) \leq \kappa |Q|$.
\end{enumerate}
Then $\nu$ is a Carleson measure with $\|\nu\|_{\cC} \leq \kappa \eta^{-1}$.
\end{lemma}

\section{Holomorphic dependency}
\label{Sec: Holomorphic dependence}

\noindent In this section we present the proof of Theorem~\ref{Thm: Holomorphic dependence} and follow an argument of Auscher and Tchamitchian for elliptic operators in divergence form~\cite[Sec.~0.5]{Auscher_Tchamitchian}. Even though the necessary modifications to the generalized Stokes operator are very small, we give the full argument for the sake of completeness. We start by proving that the set $\cO$ in Theorem~\ref{Thm: Holomorphic dependence} is open.

\begin{lemma}
    The set $\cO \coloneqq \{\mu: \text{$\mu$ satisfies Assumption~\ref{Ass: Coefficients} for some } \mu_{\bullet} , \mu^{\bullet} > 0\}$ is open in the $\L^\infty$-topology.
\end{lemma}
\begin{proof}
    Let $\mu$ satisfy Assumption~\ref{Ass: Coefficients} with constants $\mu^\bullet,\mu_\bullet>0$ and $M = (M_{\alpha \beta}^{i j})_{\alpha , \beta , i , j = 1}^d$ with $M_{\alpha \beta}^{i j} \in \L^{\infty} (\IR^d ; \IC)$ for all $1 \leq \alpha , \beta , i , j \leq d$ be such that
    \begin{align*}
        \|\mu - M \|_{\L^\infty(\IR^d;\cL(\IC^{d \times d}))} < \frac{\mu_{\bullet}}{2}.
    \end{align*}
    Then $M$ is obviously bounded and fulfills for $u \in \H^1 (\IR^d ; \IC^d)$
    \begin{align*}
        \Re \sum_{\alpha , \beta , i , j = 1}^d \int_{\IR^d} M^{i j}_{\alpha \beta} \partial_{\beta} u_j \overline{\partial_{\alpha} u_i} \, \d x &= \Re \sum_{\alpha , \beta , i , j = 1}^d \int_{\IR^d} \mu^{i j}_{\alpha \beta}\partial_{\beta} u_j \overline{\partial_{\alpha} u_i} \, \d x\\
        &\ + \Re \sum_{\alpha , \beta , i , j = 1}^d  \int_{\IR^d} (M^{i j}_{\alpha \beta}-\mu^{i j}_{\alpha \beta} ) \partial_{\beta} u_j \overline{\partial_{\alpha} u_i} \, \d x \\
        &\geq \mu_{\bullet} \| \nabla u \|_{\L^2}^2- \|M - \mu \|_{\L^\infty(\IR^d;\cL(\IC^{d \times d}))}\| \nabla u \|_{\L^2}^2\\
        &\geq  \frac{\mu_{\bullet}}{2} \| \nabla u \|_{\L^2}^2.
    \end{align*}
    Hence, $M$ satisfies Assumption~\ref{Ass: Coefficients} and the claim follows. 
\end{proof}

As in the introduction, we emphasize the dependency of $A$ on the coefficients $\mu$ by writing $A_{\mu}$. The corresponding weak Stokes operator will be denoted by $\cA_{\mu}$.

\begin{proof}[Proof of Theorem~\ref{Thm: Holomorphic dependence}]
    Let $\mu_0$ satisfy Assumption~\ref{Ass: Coefficients} and let $M = (M_{\alpha \beta}^{i j})_{\alpha , \beta , i , j = 1}^d$ with $M_{\alpha \beta}^{i j} \in \L^{\infty} (\IR^d ; \IC)$ for all $1 \leq \alpha , \beta , i , j \leq d$. By means of the previous lemma, the holomorphic map $\IC \ni z \mapsto \mu_z\coloneqq  \mu_0 + zM$ maps into $\cO$ whenever it is restricted to a small ball $|z|<\frac{\varepsilon}{\|M\|_\infty}$. For $u\in\L^2_\sigma(\IR^d)$ we have
    \begin{align*}
        ((1+t^2A_{\mu_z})^{-1} &- (1+t^2A_{\mu_0})^{-1})u \\
        &=  ((1+t^2\cA_{\mu_z})^{-1} - (1+t^2\cA_{\mu_0})^{-1})u\\
        &=  (1+t^2\cA_{\mu_0})^{-1}\big[(1+t^2\cA_{\mu_0}) -(1+t^2\cA_{\mu_z})\big] (1+t^2\cA_{\mu_z})^{-1}u\\
        &=  t^2(1+t^2\cA_{\mu_0})^{-1}\big[\cP\divergence( (\mu_0+zM) \nabla \bigdot) -\cP\divergence( \mu_0 \nabla \bigdot)\big] (1+t^2\cA_{\mu_z})^{-1}u\\
        &=  t^2\big( (1+t^2\cA_{\mu_0})^{-1}\cP\divergence\big)  zM \nabla(1+t^2A_{\mu_z})^{-1} u.
    \end{align*}
    Iterating this identity yields
    \begin{align*}
        (1+t^2A_{\mu_z})^{-1}u &= \sum_{n=0}^\infty \Big[t^2\big( (1+t^2\cA_{\mu_0})^{-1}\cP\divergence\big)  zM \nabla\Big]^n(1+t^2A_{\mu_0})^{-1} u
    \end{align*}
    which converges in $\L^2_{\sigma} (\IR^d)$ whenever $\eps > 0$ is small enough. Indeed, $(t^2\nabla (1+t^2\cA_{\mu_0})^{-1}\cP\divergence)_{t>0}$ and $(t(1+t^2\cA_{\mu_0})^{-1}\cP\divergence)_{t>0}$ are uniformly bounded in $\L^2$ by Proposition~\ref{Prop: L2 resolvent bounds} such that
    \begin{align}
    \label{eq: bounds for series}
    \begin{split}
        \Big\| \Big[t^2\big( (1+t^2\cA_{\mu_0})^{-1}\cP\divergence\big)  zM \nabla\Big]^n &(1+t^2A_{\mu_0})^{-1} u \Big\|_{\L^2} \\
        &\leq C_1^n|z|^n\|M\|^n_{\L^\infty} \|t\nabla(1+t^2A_{\mu_0})^{-1}u\|_{\L^2}\\
        &\leq  C_1^n|z|^n\|M\|^n_{\L^\infty} C\|u\|_{\L^2} \\
        &\leq 2^{- n} C \| u \|_{\L^2}
    \end{split}
    \end{align}
    for $\varepsilon \leq (2 C_1)^{-1}$ and where $C_1>0$ denotes the boundedness constant from Proposition~\ref{Prop: L2 resolvent bounds}. Now, using the Balakrishnan representation for square roots (see~\cite[Prop.~6.18]{ISEM}) we have for all $u\in \dom(A_{\mu_z}^{1/2}) =\H^1_\sigma(\IR^d)$,
    \begin{align*}
        A_{\mu_z}^\frac{1}{2}u &= \frac{2}{\pi}\int_0^\infty A_{\mu_z}(1+t^2A_{\mu_z})^{-1}u \,\d t \\
        &= \frac{2}{\pi}\int_0^\infty \frac{1}{t^2}[\Id - (1+t^2A_{\mu_z})^{-1}]u \,\d t \\
        &= \frac{2}{\pi}\int_0^\infty \frac{1}{t^2}\Big[\Id - \sum_{n=0}^\infty \Big[t^2\big( (1+t^2\cA_{\mu_0})^{-1}\cP\divergence\big)  zM \nabla\Big]^n(1+t^2A_{\mu_0})^{-1} \Big]u \,\d t \\
        &= \frac{2}{\pi}\int_0^\infty A_{\mu_0}(1+t^2A_{\mu_0})^{-1}u - \frac{1}{t^2}\sum_{n=
        1}^\infty \Big[t^2\big( (1+t^2\cA_{\mu_0})^{-1}\cP\divergence\big)  zM \nabla\Big]^n(1+t^2A_{\mu_0})^{-1} u \,\d t
    \end{align*}
    where the integrals have to be understood as improper integrals in $\L^2_\sigma(\IR^d)$. Focusing on
    \begin{align*}
        &\int_0^\infty\frac{1}{t^2}\sum_{n=
        1}^\infty \Big[t^2\big( (1+t^2\cA_{\mu_0})^{-1}\cP\divergence\big)  zM \nabla\Big]^n(1+t^2A_{\mu_0})^{-1} u \,\d t\\
        &= \lim\limits_{\delta\to 0}\int_\delta^\frac{1}{\delta} \frac{1}{t^2}\sum_{n=
        1}^\infty \Big[t^2\big( (1+t^2\cA_{\mu_0})^{-1}\cP\divergence\big)  zM \nabla\Big]^n(1+t^2A_{\mu_0})^{-1} u \,\d t,
        \intertext{we first interchange integration and summation due to~\eqref{eq: bounds for series} to get}
        &= \lim\limits_{\delta\to 0}\sum_{n=
        1}^\infty \int_\delta^\frac{1}{\delta} \frac{1}{t^2} \Big[t^2\big( (1+t^2\cA_{\mu_0})^{-1}\cP\divergence\big)  zM \nabla\Big]^n(1+t^2A_{\mu_0})^{-1} u \,\d t.
    \end{align*}
    Next, we want to interchange the limit with the series by dominated convergence but not in the space $\L^2_\sigma(\IR^d)$. Instead we will work in $\H^{-1}_\sigma(\IR^d)$ to guarantee existence of a limit in the first place. For this purpose, set for $0 < \delta \leq 1$
    \begin{align*}
        f_\delta:\IN \to \H^{-1}_\sigma(\IR^d),\quad f_\delta(n) = \int_\delta^\frac{1}{\delta} \frac{1}{t^2} \Big[t^2\big( (1+t^2\cA_{\mu_0})^{-1}\cP\divergence\big)  zM \nabla\Big]^n(1+t^2A_{\mu_0})^{-1} u \,\d t,
    \end{align*}
    then one has existence of a pointwise limit by the following argument: For $0 < \delta' < \delta \leq 1$ and $v\in \H^{1}_\sigma(\IR^d)$ one has
    \begin{align*}
        \langle f_\delta(n) - f_{\delta'}(n) , v \rangle_{\H^{-1}_\sigma , \H^{1}_\sigma} &= \int_\frac{1}{\delta}^\frac{1}{\delta'} \frac{1}{t^2} \Big\langle \Big[t^2\big( (1+t^2\cA_{\mu_0})^{-1}\cP\divergence\big)  zM \nabla\Big]^n(1+t^2A_{\mu_0})^{-1} u, v \Big\rangle_{\L^2_\sigma, \L^2_\sigma}\,\d t  \\
        &\ + \int_{\delta'}^{\delta}\frac{1}{t^2}  \Big\langle\Big[t^2\big( (1+t^2\cA_{\mu_0})^{-1}\cP\divergence\big)  zM \nabla\Big]^n(1+t^2A_{\mu_0})^{-1} u , v \Big\rangle_{\L^2_\sigma , \L^2_\sigma} \,\d t.
    \end{align*}
    Recall, that the adjoint of $A_{\mu_0}$ is the generalized Stokes operator with coefficients $\mu_0^*$. Thus, by duality and uniform boundedness of $(t^2\nabla (1+t^2\cA_{\mu_0})^{-1}\cP\divergence)_{t>0}$ and $(t(1+t^2\cA_{\mu_0})^{-1}\cP\divergence)_{t>0}$ in $\L^2$ from Proposition~\ref{Prop: L2 resolvent bounds} we estimate the integrand as follows
    \begin{align*}
        &\frac{1}{t^2} \Big| \Big\langle \Big[t^2\big( (1+t^2\cA_{\mu_0})^{-1}\cP\divergence\big)  zM \nabla\Big]^n(1+t^2A_{\mu_0})^{-1} u, v \Big\rangle_{\L^2_\sigma, \L^2_\sigma}\Big|\\
        &= \frac{1}{t^2} \Big| \Big\langle tzM \nabla\Big[t^2\big( (1+t^2\cA_{\mu_0})^{-1}\cP\divergence\big)  zM \nabla\Big]^{n-1}(1+t^2A_{\mu_0})^{-1} u, t\nabla (1+t^2A_{\mu_0}^*)^{-1}v \Big\rangle_{\L^2_\sigma, \L^2_\sigma}\Big| \\
        &\leq \frac{1}{t^2} \Big\|  tzM \nabla\Big[t^2\big( (1+t^2\cA_{\mu_0})^{-1}\cP\divergence\big)  zM \nabla\Big]^{n-1}(1+t^2A_{\mu_0})^{-1} u \Big\|_{\L^2}\cdot \| t\nabla (1+t^2A_{\mu_0}^*)^{-1}v \|_{\L^2}\\
        &\leq \frac{1}{t^2} C_1^{n-1}|z|^n\|M\|^n_{\L^\infty} \cdot \|  t\nabla(1+t^2A_{\mu_0})^{-1} u \|_{\L^2}\cdot \| t\nabla (1+t^2A_{\mu_0}^*)^{-1}v \|_{\L^2}.
    \end{align*}
    Depending on the size of $t$ we choose two different ways of controlling the gradients of the resolvents. On the one hand, Proposition~\ref{Prop: L2 resolvent bounds} implies for $u , v \in \L^2_{\sigma} (\IR^d)$
    \begin{align*}
        \| t\nabla(1+t^2A_{\mu_0})^{-1} u \|_{\L^2}\cdot \| t\nabla (1+t^2A_{\mu_0}^*)^{-1}v \|_{\L^2} \lesssim  \|u \|_{\L^2}\cdot \|v \|_{\L^2}
    \end{align*}
    and, on the other hand, the square root property yields for $u,v\in \H^1_\sigma(\IR^d)$
    \begin{align*}
        \| t\nabla(1+t^2A_{\mu_0})^{-1} u \|_{\L^2} &\cdot \| t\nabla (1+t^2A_{\mu_0}^*)^{-1}v \|_{\L^2} \\
        &\simeq \| tA_{\mu_0}^\frac{1}{2}(1+t^2A_{\mu_0})^{-1} u \|_{\L^2}\cdot \| t(A_{\mu_0}^*)^\frac{1}{2} (1+t^2A_{\mu_0}^*)^{-1}v \|_{\L^2}\\
        &=\| t(1+t^2A_{\mu_0})^{-1} A_{\mu_0}^\frac{1}{2}u \|_{\L^2}\cdot \| t (1+t^2A_{\mu_0}^*)^{-1}(A_{\mu_0}^*)^\frac{1}{2}v \|_{\L^2}\\
        &\lesssim t^2\|A_{\mu_0}^\frac{1}{2}u \|_{\L^2}\cdot \| (A_{\mu_0}^*)^\frac{1}{2}v \|_{\L^2}\\
        &\simeq t^2\|\nabla u \|_{\L^2}\cdot \| \nabla v \|_{\L^2}.
    \end{align*}
    Hence, the integrand can be controlled by $C_1^{n - 1} \lvert z \rvert^n \| M \|_{\L^{\infty}}^n \min\{1,t^{-2}\}\cdot \|u \|_{\H^1}\|v \|_{\H^1}$ which yields
    \begin{align*}
        |\langle f_\delta(n) - f_{\delta'}(n) , v \rangle_{\H^{-1}_\sigma , \H^{1}_\sigma}|\lesssim C_1^{n-1}|z|^n\|M\|^n_{\L^\infty} \|u \|_{\H^1}\|v \|_{\H^1} (\delta - \delta').
    \end{align*}
    This shows pointwise convergence of $f_\delta$ as $\delta \to 0$ and simultaneously gives a summable majorant for the sequence for $\delta\in (0,1]$, namely
    \begin{align*}
        \|f_\delta(n)\|_{\H^{-1}_\sigma(\IR^d)}\lesssim C_1^{n-1}|z|^n\|M\|^n_{\L^\infty}\|u \|_{\H^1} \leq \frac{2^{- n}}{C_1} \| u \|_{\H^1} \eqqcolon g(n).
    \end{align*}
    Thus, by dominated convergence it follows
    \begin{align*}
        &\lim\limits_{\delta\to 0}\sum_{n=
        1}^\infty \int_\delta^\frac{1}{\delta} \frac{1}{t^2} \Big[t^2\big( (1+t^2\cA_{\mu_0})^{-1}\cP\divergence\big)  zM \nabla\Big]^n(1+t^2A_{\mu_0})^{-1} u \,\d t\\
        &= \sum_{n=
        1}^\infty \lim\limits_{\delta\to 0}\int_\delta^\frac{1}{\delta} \frac{1}{t^2} \Big[t^2\big( (1+t^2\cA_{\mu_0})^{-1}\cP\divergence\big)  zM \nabla\Big]^n(1+t^2A_{\mu_0})^{-1} u \,\d t.
    \end{align*}
    Again due to Balakrishnan representation for square roots and the above identity we conclude that
    \begin{align}
    \label{eq: series for square root}
        A_{\mu_z}^\frac{1}{2}u = A_{\mu_0}^\frac{1}{2}u + \sum_{n=
        1}^\infty z^n T_n u \quad \text{in $\H^{-1}_\sigma(\IR^d)$}
    \end{align}
    holds for all $u\in \H^{1}_\sigma(\IR^d)$ and appropriate operators $(T_n)_{n\in\IN} \subseteq \cL(\H^1_{\sigma} (\IR^d) , \H^{-1}_{\sigma} (\IR^d))$. This readily proves holomorphy of the map $z \mapsto A_{\mu_z}^{1/2} u$ in the $\H^{-1}_{\sigma}$-topology. To show convergence in $\L^2_\sigma(\IR^d)$ we use Taylor's theorem (see, e.g.,~\cite[Prop.~A.1]{Vector-valued-Laplace-Trafo}) in $\H^{-1}_\sigma(\IR^d)$ to represent the coefficients as
    \begin{align*}
        T_nu = \frac{1}{2\pi \ii} \int\limits_{|w|=r} A_{\mu_w}^\frac{1}{2}u\,\frac{\d w}{w^{n+1}}
    \end{align*}
    for $r>0$ small enough. Now, due to the Kato property it is
    \begin{align*}
        \|T_n u \|_{\L^2} 
        \leq \frac{1}{2\pi}\int\limits_{|w|=r} \|A_{\mu_w}^\frac{1}{2}u\|_{\L^2}\,\frac{|\d w|}{|w|^{n+1}}
        \lesssim  \frac{1}{2\pi}\int\limits_{|w|=r} \|\nabla u\|_{\L^2}\,\frac{|\d w|}{|w|^{n+1}}
        = r^{-n}\|\nabla u\|_{\L^2}.
    \end{align*}
    Thus, the right-hand side of~\eqref{eq: series for square root} actually converges absolutely in $\L^2_\sigma(\IR^d)$ if $\lvert z \rvert < r$ and is equal to $A^{1/2}_{\mu_z} u$. Since each of the operators $A_{\mu_z}^{1/2}$ lies in $\cL(\H^1_{\sigma} (\IR^d) , \L^2_{\sigma} (\IR^d))$ the property~\cite[Prop.~A.3]{Vector-valued-Laplace-Trafo} allows to conclude holomorphy with respect to the operator norm by strong holomorphy. This proves the claim.
\end{proof}

As mentioned in the introduction, that holomorphy implies a local Lipschitz property, is well known and treated, e.\@g.\@, in~\cite[Thm.~3.24]{Morris_Turner} (see also~\cite[Thm.~2.3]{Auscher_Axelsson_McIntosh}) in the context of holomorphic functional calculus for perturbed Dirac operators. Although, as above, the changes to the square root of the generalized Stokes operator are very small, we present the full argument for the sake of completeness.    

\begin{proof}[Proof of Corollary~\ref{cor: Lipschitz estimate}]
    Let $\mu_1 \in \cO$ and $0<\delta <(\mu_1)_\bullet$. Then for all $\mu_2\in\cO$ satisfying $0 < \|\mu_2 - \mu_1 \|_{\L^\infty(\IR^d;\cL(\IC^{d \times d}))} < \delta$ we have that
    \begin{align*}
        \mu_z \coloneqq \mu_1 + z\delta \cdot \frac{\mu_2-\mu_1}{ \|\mu_1 - \mu_2 \|_{\L^\infty(\IR^d;\cL(\IC^{d \times d}))}} 
    \end{align*}
    belongs to $\cO$ if $z\in \mathbb{D} \coloneqq \{w\in\IC : |w|<1 \}$. In this case, Theorem~\ref{Thm: Kato} and Theorem~\ref{Thm: Holomorphic dependence} imply that
    \begin{align*}
        C_\delta^{-1}\|\nabla u\|_{\L^2} \leq \|A^{1/2}_{\mu_z}u \|_{\L^2} \leq C_\delta \|\nabla u\|_{\L^2} \qquad (u\in \H^1_\sigma(\IR^d))
    \end{align*}
    for some constant $C_\delta>0$ depending only on $d,(\mu_1)_\bullet, (\mu_1)^{\bullet}$ and $\delta$ and that the map
    \begin{align*}
        \mathbb{D} \ni z \mapsto A^{1/2}_{\mu_z} \in \mathcal{L}(\H^1_\sigma(\IR^d),\L^2_\sigma(\IR^d))
    \end{align*}
    is holomorphic. Next, define the function
    \begin{align*}
        G:\mathbb{D} \to \L^2_\sigma(\IR^d),\quad  G(z) \coloneqq \frac{A^{1/2}_{\mu_z}u - A^{1/2}_{\mu_1}u}{C'\|\nabla u \|_{\L^2}}
    \end{align*}
    for $u\in\H^1_\sigma(\IR^d)$ and some constant $C' >0$ to be chosen such that
    \begin{align*}
        \|G(z)\|_{\L^2} \leq \frac{\|A^{1/2}_{\mu_z}u\|_{\L^2} + \|A^{1/2}_{\mu_1}u\|_{\L^2}}{C'\|\nabla u \|_{\L^2}} \leq \frac{C_\delta + C_0}{C'}<1.
    \end{align*}
    Then for $v\in \L^2_\sigma(\IR^d)$ with $\|v\|_{\L^2} = 1$ the scalar-valued function
    \begin{align*}
        \mathbb{D}\ni z\mapsto \langle G(z) , v \rangle_{\L^2_\sigma, \L^2_\sigma}
    \end{align*}
    is holomorphic, maps onto the unit disk $\mathbb{D}$ and is equal to $0$ if $z=0$. This allows us to apply Schwarz lemma from complex analysis to get the estimate
    \begin{align*}
        \|G(z)\|_{\L^2} =\sup\limits_{\substack{v\in \L^2_\sigma(\IR^d) \\ \|v\|_{\L^2} = 1}} |\langle G(z) , v \rangle_{\L^2_\sigma , \L^2_\sigma} |\leq |z|
    \end{align*}
    for all $z\in \mathbb{D}$. Using the definition of $G(z)$ and choosing $z= \frac{\|\mu_2 - \mu_1 \|_{\L^\infty(\IR^d;\cL(\IC^{d \times d}))}}{\delta} \in \mathbb{D}$ we get
    \begin{align*}
        \|A^{1/2}_{\mu_2}u - A^{1/2}_{\mu_1}u\|_{\L^2} \leq \frac{C'}{\delta}\|\nabla u \|_{\L^2}\cdot\|\mu_2 - \mu_1 \|_{\L^\infty(\IR^d;\cL(\IC^{d \times d}))}
    \end{align*}
    for all $u\in \H^1_\sigma(\IR^d)$ which proves the claim.
\end{proof}


\end{document}